\documentclass[11pt,a4paper]{article}

\usepackage{mathrsfs} 

\usepackage{graphicx}        

\usepackage[a4paper, hmargin={2.5cm,2.5cm},vmargin={3.3cm,3.3cm}]{geometry}
\usepackage{amsmath, amssymb, amsfonts, amsbsy, latexsym, color,dsfont}
\usepackage{amscd,amsxtra}
\usepackage{amsthm}
\usepackage[normalem]{ulem}

\usepackage[T1]{fontenc}
\usepackage[latin1]{inputenc}
\usepackage[english]{babel}
\usepackage{cite}
\usepackage{paralist}
\usepackage{url}
\usepackage[bf,compact,pagestyles,small]{titlesec}
\titlespacing*{\section}{0pt}{14pt}{4pt}
\titlespacing*{\subsection}{0pt}{8pt}{3pt}
\usepackage{mathdots}
\usepackage{tikz-cd}
\usetikzlibrary{matrix,arrows,decorations.pathmorphing}
\usetikzlibrary{decorations.markings}
\tikzset{
  notarrow/.style = {decoration = {markings, mark = at position 0.5 with { \node[transform shape, xscale = 0.9, yscale=0.9,font={\sf}] {X}; } }, postaction = {decorate} },
}

\usepackage{fancyhdr,lastpage}
\def\maketimestamp{\count255=\time
\divide\count255 by 60\relax
\edef\thetime{\the\count255:}%
\multiply\count255 by-60\relax
\advance\count255 by\time
\edef\thetime{\thetime\ifnum\count255<10 0\fi\the\count255}
\edef\thedate{\number\day-\ifcase\month\or Jan\or Feb\or Mar\or
             Apr\or May\or Jun\or Jul\or Aug\or Sep\or Oct\or
             Nov\or Dec\fi-\number\year}
\def\timstamp{\hbox to\hsize{\tt\hfil\thedate\hfil\thetime\hfil}}}
\maketimestamp
\fancypagestyle{plain}{%
\fancyhf{} 
\fancyfoot[C]{\small \thepage{} of \pageref{LastPage}}}

\numberwithin{equation}{section}  
\allowdisplaybreaks[4]

\newtheorem{theorem}{Theorem}[section]

\newtheorem{lemma}[theorem]{Lemma}
\newtheorem{proposition}[theorem]{Proposition}
\newtheorem{corollary}[theorem]{Corollary}
\theoremstyle{definition}
\newtheorem{definition}[theorem]{Definition} 
\newtheorem{example}{Example}
\theoremstyle{remark}
\newtheorem{remark}{Remark}

\DeclareMathOperator{\supp}{supp} %
\DeclareMathOperator{\diagonal}{diag} %
\DeclareMathOperator*{\esssup}{ess\,sup} %
\DeclareMathOperator*{\essinf}{ess\,inf} %

\DeclareMathOperator{\epickop}{Epick}
\newcommand{\epick}[1]{\epickop{(#1)}}
\newcommand{\epickdila}{a}
\newcommand{\gti}[1][g]{\ensuremath{\cup_{j \in J}\set{T_\gamma {#1}_p}_{\gamma \in \LG_{\! j}, p \in P_j}}} 
\newcommand{\ti}[1][g]{\ensuremath{\cup_{j \in J}\set{T_\gamma {#1}_p}_{\gamma \in \LG, p \in P_j}}} 
\newcommand{\gtiE}[1][g]{\ensuremath{\cup_{j \in J}\set{T_\gamma {#1}_{j,p}}_{\gamma \in \LG_{\! j}, p \in P_j}}} 
\newcommand{\gsi}[1][g]{\ensuremath{\cup_{j \in J}\set{T_\gamma {#1}_j}_{\gamma \in \LG_{\! j}}}} 

\newcommand{\diag}[1]{\diagonal{(#1)}}

\newcommand*{\numbersys}[1]{\ensuremath{\mathbb{#1}}}
\newcommand*{\C}{\numbersys{C}}
\newcommand*{\R}{\numbersys{R}}

\newcommand*{\Q}{\numbersys{Q}}

\newcommand*{\Z}{\numbersys{Z}}

\newcommand*{\N}{\numbersys{N}}
\newcommand*{\T}{\numbersys{T}}

\newcommand*{\cH}{\mathcal{H}}
\newcommand*{\cP}{\mathcal{P}}
\newcommand{\itvoc}[2]{\ensuremath{\left({#1},{#2}\right]}} 
\newcommand{\itvco}[2]{\ensuremath{\left[{#1},{#2}\right)}} %
\newcommand{\abs}[1]{\ensuremath{\left\lvert#1\right\rvert}}
\newcommand{\abssmall}[1]{\ensuremath{\lvert#1\rvert}}
\newcommand{\absbig}[1]{\ensuremath{\bigl\lvert#1\bigr\rvert}}

\newcommand{\norm}[2][]{\ensuremath{\left\lVert#2\right\rVert_{#1}}}

\newcommand{\normsmall}[2][]{\ensuremath{\lVert#2\rVert_{#1}}}
\newcommand{\innerprod}[3][]{\ensuremath{\left\langle #2,#3\right\rangle_{\! #1}}}

\newcommand{\set}[1]{\ensuremath{\left\lbrace{#1}\right\rbrace}}

\newcommand{\setprop}[2]{\ensuremath{\left\lbrace{#1} : {#2}\right\rbrace}}

\newcommand{\setsmall}[1]{\ensuremath{\lbrace{#1}\rbrace}}
\newcommand{\seqsmall}[1]{\ensuremath{\lbrace{#1}\rbrace}}
\newcommand{\setpropbig}[2]{\ensuremath{\bigl\lbrace{#1} :
      {#2}\bigr\rbrace}}
\newcommand{\setpropsmall}[2]{\ensuremath{\lbrace{#1} : {#2}\rbrace}}

\newcommand{\lat}[1]{\ensuremath {#1}} 
\newcommand{\LG}{\ensuremath\lat{\Gamma}}
\newcommand{\LL}{\ensuremath\lat{\Lambda}}
\newcommand\cD{\mathcal{D}} 
\newcommand\sumstyle[2]{#1: #2} 

\usepackage{xspace}
\newcommand{\ie}{i.e.,\xspace} 
\newcommand{\eg}{e.g.,\xspace} 

\newcommand{\etc}{etc.\@\xspace}

\newcommand{\almoste}{\text{a.e.}}

\newcommand{\cfm}{\mu_M} 
\newlength{\dhatheight}
\newcommand{\doublehat}[1]{%
	\settoheight{\dhatheight}{\ensuremath{\widehat{#1}}}%
	\addtolength{\dhatheight}{-0.35ex}%
    	\widehat{\vphantom{\rule{1pt}{\dhatheight}}%
    	\smash{\hspace{-2pt} \widehat{#1}}}}
\newcommand{\ghat}{\widehat{G}}
\newcommand{\ghhat}{\doublehat{G}}
\newcommand*\oline[1]{%
  \vbox{%
    \hrule height 0.5pt
    \kern0.25ex
    \hbox{%
      \kern-0.1em
      \ifmmode#1\else\ensuremath{#1}\fi
      \kern-0.1em
    }
  }
}

{\nopagebreak\par\noindent\hbox to \hsize{\hrulefill}\par\noindent}

\usepackage{hyperref}
\hypersetup{
 pdfview={FitH},
 pdfstartview={FitH},
 pdfauthor = {Mads Sielemann Jakobsen, Jakob Lemvig},
 pdftitle = {Reproducing formulas for generalized translation invariant systems on locally compact abelian groups}, 
 pdfsubject = {harmonic analysis},
 pdfkeywords = {generalized shift invariant systems, characterizing equations, dual frames},
 pdfcreator = {LaTeX with hyperref package},
 pdfproducer = PDFlatex}

\makeatletter
\def\blfootnote{\xdef\@thefnmark{}\@footnotetext} 
\def\subjclass{\xdef\@thefnmark{}\@footnotetext}
\long\def\symbolfootnote[#1]#2{\begingroup%
\def\thefootnote{\fnsymbol{footnote}}\footnote[#1]{#2}\endgroup} 
\if@titlepage
  \renewenvironment{abstract}{%
      \titlepage
      \null\vfil
      \@beginparpenalty\@lowpenalty
      \begin{center}%
        \bfseries \abstractname
        \@endparpenalty\@M
      \end{center}}%
     {\par\vfil\null\endtitlepage}
\else
  \renewenvironment{abstract}{%
      \if@twocolumn
        \section*{\abstractname}%
      \else
        \small
        \list{}{%
          \settowidth{\labelwidth}{\textbf{\abstractname:}}
          \setlength{\leftmargin}{50pt}
          \setlength{\rightmargin}{50pt}
          \setlength{\itemindent}{\labelwidth}
          \addtolength{\itemindent}{\labelsep}
        }
        \item[\textbf{\abstractname:}]

      \fi}
      {\if@twocolumn\else\endlist\fi}
\fi
\makeatother

\begin{document}

\title{Reproducing formulas for generalized translation invariant systems on locally compact abelian groups}

\date{February 27, 2015}

 \author{Mads Sielemann Jakobsen\footnote{Technical University of Denmark, Department of Applied Mathematics and Computer Science, Matematiktorvet 303B, 2800 Kgs.\ Lyngby, Denmark, E-mail: \protect\url{msja@dtu.dk}}
, Jakob Lemvig\footnote{Technical University of Denmark, Department of Applied Mathematics and Computer Science, Matematiktorvet 303B, 2800 Kgs.\ Lyngby, Denmark, E-mail: \protect\url{jakle@dtu.dk}}}

 \blfootnote{2010 {\it Mathematics Subject Classification.} Primary: 42C15, 43A32, 43A70, Secondary: 43A60, 46C05.} \blfootnote{{\it Key
     words and phrases.} continuous frame, dual frames, dual generators, g-frame, Gabor frame, generalized shift invariant system, generalized translation invariant system, LCA group, Parseval frame, wavelet frame}

\maketitle 

\thispagestyle{plain}
\begin{abstract} 
  In this paper we connect the well established discrete frame theory of generalized shift invariant systems to a continuous frame theory. To do so, we let $\LG_{\!j}$, $j \in J$, be a countable family of closed, co-compact subgroups of a second countable locally compact abelian group $G$ and study systems of the form $\cup_{j \in J}\{ g_{j,p}(\cdot - \gamma)\}_{\gamma\in \LG_{\!j}, p \in P_j}$ with generators $g_{j,p}$ in $L^2(G)$ and with each $P_j$ being a countable or an uncountable index set. We refer to systems of this form as generalized translation invariant (GTI) systems. Many of the familiar transforms, \eg the wavelet, shearlet and Gabor transform, both their discrete and continuous variants, are GTI systems. Under a technical $\alpha$ local integrability condition ($\alpha$-LIC) we characterize when GTI systems constitute tight and dual frames that yield reproducing formulas for $L^2(G)$. This generalizes results on generalized shift invariant systems, where each $P_j$ is assumed to be countable and each $\LG_{\!j}$ is a uniform lattice in $G$, to the case of uncountably many generators and (not necessarily discrete) closed, co-compact subgroups. Furthermore, even in the case of uniform lattices $\LG_{\! j}$, our characterizations improve known results since the class of GTI systems satisfying the $\alpha$-LIC is strictly larger than the class of GTI systems satisfying the previously used local integrability condition. As an application of our characterization results, we obtain new characterizations of translation invariant continuous frames and Gabor frames for $L^2(G)$.  In addition, we will see that the admissibility conditions for the continuous and discrete wavelet and Gabor transform in $L^2(\R^n)$ are special cases of the same general characterizing equations.
\end{abstract}

\section{Introduction}
\label{sec:introduction}
In harmonic analysis one is often interested in determining conditions on generators of function
systems, \eg Gabor and wavelet systems, that allow for reconstruction of any function in a given
class of functions from its associated transform via a reproducing formula. The work of Hern\'andez,
Labate, and Weiss~\cite{MR1916862} and of Ron and Shen~\cite{MR2132766} on generalized shift
invariant systems in $L^2(\R^n)$ presented a unified theory for many of the familiar \emph{discrete}
transforms, most notably the Gabor and the wavelet transform. The generalized shift invariant
systems are collections of functions of the form $\gsi$, where $J$ is a countable index set,
$T_\gamma$ denotes translation by $\gamma$, $\LG_{\!j}$ a full-rank lattice in $\R^n$, and
$\{g_j\}_{j \in J}$ a subset of $L^2(\R^n)$. Here, the word ``shift'' is used since the translations
are discrete and the word ``generalized'' since the shift lattices $\LG_{\!j}$ are allowed to change
with the parameter $j\in J$. The main result of Hern\'andez, Labate, and Weiss~\cite{MR1916862} is a
characterization, by so-called $t_\alpha$-equations, of all functions $g_j$ that give rise to
isometric transforms, called Parseval frames in frame theory.

The goal of this work is to connect the discrete transform theory of generalized shift invariant
systems to a continuous/integral transform theory. In doing so, the scope of the ``unified
approach'' started in \cite{MR1916862,MR2132766} will be vastly extended.  What more is, this new
theory will cover ``intermediate'' steps, the semi-continuous transforms, and we will do so in a
very general setting of square integrable functions on locally compact abelian groups. In
particular, we recover the usual characterization results for discrete and continuous Gabor and
wavelet systems as special cases. For discrete wavelets in $L^2(\R)$ with dyadic dilation, this
result was obtained in 1995, independently by Gripenberg~\cite{MR1338828} and Wang~\cite{MR2692675},
and it can be stated as follows. Define the translation operator $T_bf(x)=f(x-b)$ and dilation
operator $D_af(x)=\abs{a}^{-1/2}f(x/a)$ for $b\in \R, a\neq 0$. The discrete wavelet system
$\seqsmall{T_{2^{j} k}D_{2^{j}}\psi}_{j,k\in \Z}$ generated by $\psi \in L^2(\R)$ is indeed a
generalized shift invariant system with $J=\Z$, $\LG_{\!j}=2^{j}\Z$, and $g_j=D_{2^{j}} \psi$. Now,
the linear operator $W_d$ defined by
\[ 
W_d : L^2(\R) \to \ell^2(\Z^2), \quad W_d f(j,k) = \innerprod{f}{T_{2^{j} k}D_{2^{j}}\psi} 
\] 
is isometric if, and only if, for all $\alpha \in \bigcup_{j\in \Z} 2^{-j}\Z $, the following
$t_{\alpha}$-equations 
hold:
\begin{equation}
t_\alpha:= \sum_{j\in \Z \, : \, \alpha \in 2^{-j}\Z} \widehat \psi(2^j \xi)
\overline{\widehat\psi(2^j(\xi+\alpha))} = \delta_{\alpha,0} \quad \text{for a.e. } \xi \in \widehat{\R},
\label{eq:t-alpha-discrete-wavelet-intro}
\end{equation}
where $\widehat{\R}$ denotes the Fourier domain. In the language of frame theory, we say that
generators $\psi \in L^2(\R)$ of discrete Parseval wavelet frames have been characterized by
$t_\alpha$-equations.

Calder\'on~\cite{MR0167830} discovered in 1964 that any function $\psi \in L^2(\R)$ satisfying the
Calder\'on admissibility condition
\begin{equation}
\int_{\R \setminus \set{0}} \frac{\abssmall{\widehat\psi(a\xi)}^2}{\abs{a}} \, da =1  \quad \text{for a.e. } \xi \in \widehat{\R}
\label{eq:t-alpha-cont-wavelet-intro}
\end{equation}
leads to reproducing formulas for the \emph{continuous} wavelet transform. To be precise, the
linear operator $W_c$ defined by
\[ 
W_c : L^2(\R) \to L^2\big(\R\!\setminus\!\set{0}\times \R,\tfrac{dadb}{a^2}\big), \quad W_c f(a,b) =
\innerprod{f}{T_{b}D_{a}\psi}
\] 
is isometric if, and only if, the Calder\'on admissibility condition holds.  We will see
that the Calder\'on admissibility condition is nothing but the $t_\alpha$-equation (there is only
one!) for the continuous wavelet system. Similar results hold for the Gabor case; here the
continuous transform is usually called the short-time Fourier transform. Actually, the theory is not
only applicable to the Gabor and wavelet setting, but to a very large class of systems of functions
including shearlet and wave packet systems, which we shall call \emph{generalized translation
  invariant} systems. We refer the reader to the classical texts
\cite{gr01,Daubechies:1992:TLW:130655,MR507783} and the recent book \cite{MR2896273} for
introductions to the specific cases of Gabor, wavelet, shearlet and wave packet analysis.

In~\cite{MR2283810}, Kutyniok and Labate generalized the results of Hern\'andez, Labate, and Weiss
to generalized shift invariant systems $\gsi$ in $L^2(G)$, where $G$ is a second countable locally
compact abelian group and $\LG_{\!j}$ is a family of uniform lattices (\ie $\LG_{\!j}$ is a discrete
subgroup and the quotient group $G/\LG_{\!j}$ is compact) indexed by a countable set $J$. The main
goal of the present paper is to develop the corresponding theory for \emph{semi-continuous} and
\emph{continuous} frames in $L^2(G)$. In order to achieve this, we will allow non-discrete
translation groups $\LG_{\!j}$, and we will allow for each translation group to have uncountable
many generators, indexed by some index set $P_j$, $j\in J$.  We say that the corresponding family
$\gtiE$ in $L^2(G)$ is a \emph{generalized translation invariant system}. To be precise, we will,
for each $j \in J$, take $P_j$ to be a $\sigma$-finite measure space with measure $\mu_{P_j}$ and
$\LG_{\!j}$ to be closed, co-compact (\ie the quotient group $G/\LG_{\!j}$ is compact) subgroups. We
mention that \emph{any} locally compact abelian group has a co-compact subgroup, namely the group
itself. On the other hand, there exist groups that do not contain uniform lattices, \eg the $p$-adic
numbers. Thus, the theory of generalized translation invariant systems is applicable to a larger
class of locally compact abelian groups than the theory of generalized shift invariant systems.
 
The two wavelet cases described above fit our framework. The discrete wavelet system can be written
as $\cup_{j \in \Z}\{T_\gamma (D_{2^{j}}\psi)\}_{\gamma \in 2^{j}\Z}$, so we see that $P_j$ is a
singleton and $\mu_{P_j}$ a weighted counting measure for each $j \in J = \Z$, and that there are
countably many different (discrete) $\LG_{\!j}$. For the continuous wavelet system on the form
$\set{T_\gamma(D_p \psi)}_{\gamma \in \R,p \in \R\setminus\{0\}}$, we have that $J$ is a singleton,
\eg $\set{{j_0}}$ since there is only one translation subgroup $\LG_{\!{j_0}}=\R$. On the other
hand, here ${P_{j_0}}$ is uncountable and $\mu_{P_{j_0}}$ a weighted Lebesgue measure.  We stress
that our setup can handle countable many (distinct) $\LG_{\!j}$ and countable many ${P_j}$, each
being uncountable.

The characterization results in \cite{MR1916862,MR2283810} rely on a technical condition on the
generators and the translation lattices, the so-called \emph{local integrability condition}. This
condition is straightforward to formulate for generalized translation invariant systems, however, we
will replace it by a strictly weaker condition, termed \emph{$\alpha$ local integrability
  condition}. Therefore, even for generalized \emph{shift} invariant systems in the euclidean
setting, our work extends the characterization results by Hern\'andez, Labate, and
Weiss~\cite{MR1916862}. Under the $\alpha$ local integrability condition, we show in
Theorem~\ref{th:parseval-char-LCA} that $\gtiE$ is a Parseval frame for $L^2(G)$, that is, the
associated transform is isometric if, and only if,
\begin{equation*}
t_{\alpha} := \sum_{j \in J \, : \, \alpha \in \Gamma_{\!j}^{\perp}} \int_{P_j}
\overline{\hat{g}_{j,p}(\omega)} \hat{g}_{j,p}(\omega+\alpha ) \, d\mu_{{P_j}}(p) =
\delta_{\alpha,0} \quad \text{a.e. } \omega \in \ghat
\end{equation*}
for every $\alpha \in \cup_{j\in J} \LG_{\!j}^{\perp}$, where $\LG_{\!j}^{\perp}=\big\{\omega \in
\widehat{G} \, : \, \omega(x) = 0 \text{ for all } x\in \LG_{\!j}\big\}$ denotes the annihilator of
$\LG_{\!j}$. Now, returning to the two main examples of this introduction, the discrete and
continuous wavelet transform, we see why the number of the $t_\alpha$-equations in
\eqref{eq:t-alpha-discrete-wavelet-intro} and \eqref{eq:t-alpha-cont-wavelet-intro} are so
different.  In the discrete case the corresponding union of the annihilators of the translation
groups is $\cup_{j\in \Z} 2^{-j}\Z$, while in the continuous case the annihilator of $\R$ is simply
$\set{0}$, which corresponds to only one $t_\alpha$-equation ($\alpha=0$).

Finally, as Kutyniok and Labate~\cite{MR2283810} restrict their attention to Parseval frames, there
are currently no characterization results available for \emph{dual} (discrete) frames in the setting
of locally compact abelian groups. Hence, one additional objective of this paper is to prove
characterizing equations for \emph{dual} generalized translation invariant frames to remedy this
situation.

For a related study of reproducing formulas from a purely group representation theoretical point of
view, we refer to the work of F\"uhr~\cite{MR2652610}, and De Mari, De Vito \cite{MR3008561}, and
the references therein.

\medskip

\noindent The paper is organized as follows. We recall some basic theory about locally compact
abelian groups and introduce the generalized translation invariant systems in Section~\ref{sec:LCA}
and~\ref{sec:gener-transl-invar}, respectively. Additionally, in Section \ref{sec:frame-theory} we
give a short introduction to the theory of continuous frames and g-frames. In Section \ref{sec:gsi}
we present our main characterization result for dual generalized translation invariant frames
(Theorem~\ref{th:2001e}) and, as corollary, then for Parseval frames
(Theorem~\ref{th:parseval-char-LCA}). In Section~\ref{sec:suff-cond-local} and
\ref{sec:exampl-local-integr} we relate several conditions used in our main results. Finally, we
consider the special case of translation invariant systems and apply our characterization results on
concrete groups and to concrete examples in Sections \ref{sec:special-syst} and
\ref{sec:applications}. Specifically, we consider discrete and continuous wavelet systems in
$L^2(\R^n)$, shearlets in $L^2(\R^2)$, discrete, semi-continuous and continuous Gabor frames on LCA
groups and GTI systems over the $p$-adic integers and numbers.

During the final stages of this project, we realized that Bownik and Ross~\cite{BowRos2014} have
completed a related investigation. As they consider and characterize the structure of translation
invariant subspaces on locally compact abelian groups, their results do not overlap with our results
in any way. However, they do consider translations along a closed, co-compact subgroup. We adopt
their terminology of \emph{translation} invariance, in place of shift invariance, to emphasize the
fact that $\LG_{\!j}$ need not be discrete.

\section{Preliminaries} \label{sec:preliminaries}

In the following sections we set up notation and recall some useful results from Fourier analysis on
locally compact abelian (LCA) groups and continuous frame theory. Furthermore, we will prove two
important lemmas, Lemma~\ref{le:1902a} and \ref{thm:unitary-equi-duals}.

\subsection{Fourier analysis on locally compact abelian groups} \label{sec:LCA}

Throughout this paper $G$ will denote a second countable locally compact abelian group.  We note
that the following statements are equivalent: (i) $G$ is second countable, (ii) $L^2(G)$ is
separable, (iii) $G$ is metrizable and $\sigma$-compact. Note that the metric on $G$ can be chosen
to be translation invariant.

To $G$ we associate its dual group $\ghat$ consisting of all characters, \ie all continuous
homomorphisms from $G$ into the torus $\T \cong \setprop{z\in\C}{ \abs{z} =1}$. Under pointwise
multiplication $\ghat$ is also a locally compact abelian group. We will use addition and
multiplication as group operation in $G$ and $\ghat$, respectively. Note that in the introduction we
used addition as group operation in $\ghat$.  By the Pontryagin duality theorem, the dual group of
$\ghat$ is isomorphic to $G$ as a topological group, \ie $\ghhat \cong G$.  We recall the well-known
facts that if $G$ is discrete, then $\ghat$ is compact, and vice versa.

We denote the Haar measure on $G$ by $\mu_G$. The (left) Haar measure on any locally compact group
is unique up to a positive constant. From $\mu_G$ we define $L^1(G)$ and the Hilbert space $L^2(G)$
over the complex field in the usual way.

For functions $f\in L^1(G)$ we define the Fourier transform
\[ 
\mathcal Ff(\omega) = \hat f(\omega) = \int_{G} f(x)  \overline{\omega(x)} \, d\mu_G(x), \quad
\omega \in \ghat.
\] 
If $f\in L^1(G), \hat{f} \in L^1(\ghat)$, and the measure on $G$ and $\ghat$ are normalized so that
the Plancherel theorem holds (see \cite[(31.1)]{MR0262773}), the function $f$ can be recovered from
$\hat{f}$ by the inverse Fourier transform
\[
f(x) = \mathcal F^{-1}\hat{f}(x) = \int_{\ghat} \hat{f}(\omega) \omega(x) \, d\mu_{\ghat}(\omega),
\quad x\in G.
\]
From now on we always assume that the measure on a group $\mu_G$ and its dual group $\mu_{\ghat}$
are normalized this way, and we refer to them as \emph{dual measures}.  As in the classical Fourier
analysis $\mathcal F$ can be extended from $L^1(G) \cap L^2(G)$ to an isometric isomorphism between
$L^2(G)$ and $L^2(\ghat)$.

On any locally compact abelian group $G$, we define the following two linear operators.  For $a \in G$, the
operator $T_{a}$, called \emph{translation} by $a$, is defined by
\[
 T_{a}:L^2(G)\to L^2(G), \ (T_{a}f)(x) = f(x-a), \quad x\in G.
\]
For $\chi\in \ghat$, the
operator $E_{\chi}$, called \emph{modulation} by $\chi$, is defined by
\[ 
E_{\chi}:L^2(G)\to L^2(G), \ (E_{\chi}f)(x) = \chi(x) f(x), \quad x\in G.
\]
Together with the Fourier transform $\mathcal F$, the two operators $E_{\chi}$ and $T_a$ share the
following commutator relations: $T_aE_{\chi} = \overline{\chi(a)} E_{\chi}T_a , \ \mathcal F T_a =
E_{a^{-1}} \mathcal F$, and $ \mathcal F E_{\chi} = T_{\chi} \mathcal F.$

For a subgroup $H$ of an LCA group $G$, we define its annihilator as
\[ 
H^{\perp} = \{ \omega \in \ghat \, : \, \omega(x) = 1 \ \text{for all} \ x\in H \}.
\] 
The annihilator $H^{\perp}$ is a closed subgroup in $\ghat$, and if $H$ is closed, then $\widehat H
\cong \ghat / H^{\perp}$ and $\widehat{G/H} \cong H^{\perp}$.

We will repeatedly use Weil's formula; it relates integrable functions over $G$ with integrable
functions on the quotient space $G/H$ when $H$ is a closed  subgroup of $G$. We mention the
following results concerning Weil's formula \cite{MR1802924}.

\begin{theorem} \label{th:1802f} Let $H$ be a closed subgroup of $G$. Let $\pi_H:G\to G/H, \
  \pi_H(x) = x+H$ be the \textit{canonical map} from $G$ onto $G/H$. If $f\in L^1(G)$, then the
  following holds:
\begin{enumerate}[(i)]
\item The function $\dot x \mapsto \int_H f(x+h) \, d\mu_{H}(h)$, $\dot x = \pi_H(x)$ defined almost everywhere on $G/H$, is integrable.
\item (Weil's formula) Let two of the Haar measures on $G, H$ and $G/H$ be given, then the third can be normalized such that 
\begin{equation}
\label{eq:1802a}  \int_G f(x) \, d\mu_{G}(x) = \int_{G/H} \int_H f(x+h) \, d\mu_{H}(h) \, d\mu_{G/H}(\dot x) .
\end{equation} 
\item If \eqref{eq:1802a} holds, then the respective dual measures on $\widehat G, H^{\perp}\cong
  \widehat{G/H}$, $\ghat / H^{\perp} \cong \widehat H$ satisfy
\begin{equation}
  \label{eq:1802b} 
  \int_{\widehat G} \hat{f}(\omega) \, d\mu_{\ghat}(\omega) = \int_{\ghat/H^{\perp}} \int_{H^{\perp}} \hat{f}(\omega \gamma) \, d\mu_{H^{\perp}}(\gamma )\, d\mu_{\ghat/H^{\perp}}(\dot \omega).
\end{equation}
\end{enumerate}
\end{theorem}
\begin{remark} \label{rem:1802f} 
  Since a Haar measure and its dual are chosen so that the Plancherel theorem holds we have the
  following uniqueness result: If two of the measures on $G,H,G/H,\ghat, H^{\perp}$ and $\widehat G
  / H^{\perp}$ are given, and these two are not dual measures, by requiring Weil's formulas
  \eqref{eq:1802a} and \eqref{eq:1802b}, all other measures are uniquely determined.
\end{remark}
For more information on harmonic analysis on locally compact abelian groups, we refer the reader to
the classical books \cite{MR1397028,MR0156915,MR0262773,MR1802924}. 

\smallskip

For a Borel set $E \subset \ghat$ with $\mu_{\ghat}{(\oline{E})}=0$, we define:
\begin{equation}
  \label{eq:def-D}
  \cD = \big\{ f \in L^2 (G) \, : \, \hat{f} \in L^\infty(\ghat) \text{ and }
    \supp \hat{f} \text{ is compact in } \ghat \setminus E \big\}.
\end{equation}
It is not difficult to show that $\cD$ is dense in $L^2(G)$ exactly when
$\mu_{\ghat}{(\oline{E})}=0$. We will frequently prove our results on $\cD$ and extend by a density
argument. The role of the set $E$ is to allow for ``blind spots'' of transforms -- a term coined by
F\"{u}hr~\cite{Fuhr2013}.  We will let $E$ be an unspecified set satisfying
$\mu_{\ghat}{(\oline{E})}=0$; the specific choice of $E$ depends on the application, \eg in the
Gabor and wavelet case~\cite{MR1916862} one would usually take $E=\emptyset$ and $E=\set{0}$,
respectively.

The following result relies on Weil's formula and will play an important part of the proofs in
Section~\ref{sec:gsi}.
\begin{lemma} \label{le:1902a} Let $H$ be a closed subgroup of an LCA group $G$ with Haar measure
  $\mu_H$. Suppose that $f_1,f_2 \in \mathcal D$ and $\varphi,\psi\in L^2(G)$. Then
  \[ 
\int_H \langle f_1, T_h\varphi\rangle \langle T_h\psi,f_2\rangle \, d\mu_H(h) = \int_{\ghat}
  \int_{H^{\perp}} \hat{f_1}(\omega) \overline{\hat{f}_2(\omega\alpha)}
  \overline{\hat{\varphi}(\omega)} \hat{\psi} (\omega\alpha) \, d\mu_{H^{\perp}}(\alpha) \,
  d\mu_{\ghat}(\omega).
  \]
\end{lemma}
\begin{proof}
  Let $h\in H$. An application of the Plancherel theorem together with Weil's formula yields
\begin{align*}
  \langle f_1, T_h \varphi\rangle & = \langle \hat{f_1}, \widehat{T_h\varphi} \rangle = \langle \hat{f}_1, E_{-h} \hat{\varphi}\rangle = \int_{\ghat} \hat{f_1}(\omega) \overline{\hat{\varphi}(\omega)} \omega(h) \, d\mu_{\ghat} (\omega) \\
  & = \int_{\ghat / H^{\perp}} \int_{H^{\perp}} \hat{f}_1(\omega \gamma)
  \overline{\hat{\varphi}(\omega \gamma)} \omega(h) \gamma(h) \, d\mu_{H^{\perp}} (\gamma) \,
  d\mu_{\ghat/H^{\perp}} (\dot{\omega}) \\ & = \int_{\widehat H} \Big( \omega (h) \int_{H^{\perp}}
  \hat{f}_1(\omega \gamma) \overline{\hat{\varphi}(\omega \gamma)} \, d\mu_{H^{\perp}} (\gamma)
  \Big) \, d\mu_{\widehat H} (\omega),
\end{align*}
where we tacitly used that $\ghat/ H^{\perp}\cong \widehat{ H}$.  A similar calculation can be done
for $\langle T_h\psi,f_2\rangle$. To ease notation, we define $[\hat{f},
\hat{\varphi}](\omega,H^{\perp})=\int_{H^{\perp}} \hat{f}(\omega \gamma)
\overline{\hat{\varphi}(\omega \gamma)} \, d\mu_{H^{\perp}} (\gamma)$ for $f\in\mathcal D$.  Again,
by the Plancherel theorem and Weil's formula we have
\begin{align*}
  & \quad \ \int_H \langle f_1, T_h\varphi\rangle \langle T_h\psi,f_2\rangle \, d\mu_H(h) \\ & = \int_H \bigg( \int_{ \widehat H} \omega (h)  [\hat{f_1}, \hat{\varphi}](\omega,H^{\perp}) \ d\mu_{\widehat H}(\omega) \bigg) \overline{\bigg( \int_{ \widehat H} \omega (h) [\hat{f_2}, \hat{\psi}](\omega,H^{\perp}) \ d\mu_{\widehat H}(\omega) \bigg)}  \, d\mu_H(h)\\
  & = \Big\langle \mathcal F^{-1} [\hat{f}_1, \hat{\varphi}](\,\cdot\, ,H^{\perp})  ,  \mathcal F^{-1} [\hat{f_2}, \hat{\psi}](\,\cdot\, ,H^{\perp}) \Big\rangle_{L^2(H)} \\
  & = \Big\langle [\hat{f_1}, \hat{\varphi}](\,\cdot\,,H^{\perp})  ,   [\hat{f_2}, \hat{\psi}](\,\cdot\,,H^{\perp}) \Big\rangle_{L^2(\widehat H)} \\
  & = \int_{\ghat / H^{\perp}} \Big(\int_{H^{\perp}} \hat{f_1}( \omega \gamma)
  \overline{\hat{\varphi}(\omega \gamma)} \, d\mu_{H^{\perp}}(\gamma) \Big) \overline{\Big(
    \int_{H^{\perp}} \hat{f_2}( \omega \gamma)
    \overline{\hat{\psi}(\omega \gamma)} \, d\mu_{H^{\perp}}(\gamma) \Big)} \, d\mu_{\ghat/H^{\perp}}(\dot{\omega}) \\
  & = \int_{\ghat / H^{\perp}} \Big[ \int_{H^{\perp}} \hat{f_1}( \omega \gamma)
  \overline{\hat{\varphi}(\omega \gamma)} \Big( \int_{H^{\perp}} \overline{\hat{f_2}( \omega \beta)}
  \hat{\psi}(\omega \beta)
  \, d\mu_{H^{\perp}}(\beta) \Big) \, d\mu_{H^{\perp}}(\gamma) \Big] \, d\mu_{\ghat/H^{\perp}}(\dot{\omega}) \\
  & = \int_{\ghat / H^{\perp}} \Big[ \int_{H^{\perp}} \hat{f_1}( \omega \gamma)
  \overline{\hat{\varphi}(\omega \gamma)} \Big( \int_{H^{\perp}} \overline{\hat{f_2}( \omega \gamma
    \alpha)} \hat{\psi}(\omega \gamma \alpha)
  \, d\mu_{H^{\perp}} (\alpha) \Big) \, d\mu_{H^{\perp}} (\gamma) \Big] \, d\mu_{\ghat/H^{\perp}}(\dot{\omega}) \\
  & = \int_{\ghat} \hat{f_1}( \omega ) \overline{\hat{\varphi}(\omega )} \Big( \int_{H^{\perp}}
  \overline{\hat{f_2}( \omega \alpha)}
  \hat{\psi}(\omega \alpha) \, d\mu_{H^{\perp}}(\alpha) \Big) \, d\mu_{\ghat}(\omega) \\
  & = \int_{\ghat} \int_{H^{\perp}} \hat{f_1}( \omega ) \overline{\hat{f_2}( \omega \alpha)}
  \overline{\hat{\varphi}(\omega )} \hat{\psi}(\omega \alpha) \, d\mu_{H^{\perp}}(\alpha) \,
  d\mu_{\ghat}(\omega).
\end{align*}
Here $\mathcal F$ denotes the Fourier transform on $H$.
\end{proof}

\subsection{Definition of generalized translation invariant systems}
\label{sec:gener-transl-invar}
 
Let $J \subset \Z$ be a countable index set. For each $j \in J$, let $P_j$ be a countable or an
uncountable index set, let $g_{j,p} \in L^2(G)$ for $p \in P_j$, and let $\LG_{\!j}$ be a closed,
co-compact subgroup in $G$. Recall that co-compact subgroups are subgroups of $G$ for which
$G/\Gamma_{\!j}$ is compact. For a compact abelian group, the group is metrizable if, and only if,
the character group is countable \cite[(24.15)]{MR0156915}. Hence, since $G/\Gamma_{\! j}$ is
compact and metrizable, the group $\widehat{G/\Gamma}_{\! j} \cong \Gamma_{\! j}^{\perp}$ is
discrete and countable. Unless stated otherwise we equip $\Gamma_{\! j}^{\perp}$ with the counting
measure and assume a fixed Haar measure $\mu_G$ on $G$. By Remark \ref{rem:1802f} this uniquely
determines the measures on $\Gamma_{\! j}, G/\Gamma_{\! j}, \ghat$, and $\ghat/\Gamma_{\!
  j}^{\perp}$.

The \emph{generalized translation invariant} (GTI) system generated by $\{g_{j,p}\}_{p\in P_{j},j\in
  J}$ with translation along closed, co-compact subgroups $\{\LG_{\!j}\}_{j \in J}$ is the family of
functions $\gtiE$. To ease notation, we will suppress the dependence of $j$ in $g_{j,p}$ and write
the GTI system as $\gti$.

If we take $\Gamma=\LG_{\!j}$ for each $j\in J$, we obtain a translation invariant (TI) system in
the sense that $f \in\ti$ implies $T_\gamma f \in\ti$ for all $\gamma \in \LG$. However, generalized
translation invariant systems are more general than translation invariant systems since we allow for
a different subgroup for each set of generators $\set{g_{j,p}}_{p \in P_j}$.

When each $P_j$ is countable and each $\LG_{\!j}$ is a uniform lattice, i.e., a \emph{discrete},
co-compact subgroup, we recover the generalized shift invariant (GSI) systems considered in
\cite{MR2283810}.  However, we note that there exist locally compact abelian groups that do not
contain any uniform lattices. As an example we mention the $p$-adic numbers, whose only discrete
subgroup is the neutral element which is not a uniform lattice. In other cases, such as the $p$-adic
integers, the LCA group will have only trivial examples of uniform lattices, \eg the neutral
element, but have plenty non-trivial co-compact subgroups, see Example~\ref{ex:p-adic} in
Section~\ref{sec:applications}.

Finally, as an alternative generalization of uniform lattices, we mention the idea of so-called
quasi-lattices, see \cite{MR2442085,MR2377132}. In contrast to closed, co-compact subgroups,
quasi-lattices are discrete subsets in $G$ that are not necessarily groups.

\subsection{Frame theory}
\label{sec:frame-theory}
The central concept of this section is that of a continuous frame. The definition is as follows.
\begin{definition} \label{def:cont-frames} Let $\cH$ be a complex Hilbert space, and
  let $(M,\Sigma_M,\cfm)$ be a measure space, where $\Sigma_M$ denotes the $\sigma$-algebra and
  $\cfm$ the non-negative measure. A family of vectors $\set{f_k}_{k \in M}$ is called a
  \emph{continuous frame} for $\cH$ with respect to $(M,\Sigma_M,\cfm)$ if
  \begin{enumerate}[(a)]
  \item $k\mapsto f_k$ is weakly measurable, \ie for all $f \in \cH$,
    the mapping $M \to \C, k \mapsto \innerprod{f}{f_k}$ is
    measurable, and
  \item there exist constants $A,B>0$ such that
    \begin{equation}
      \label{eq:cont-frame-inequality} 
      A \norm{f}^2 \le \int_M \abs{\innerprod{f}{f_k}}^2 d\cfm(k) 
      \le B \norm{f}^2 \quad \text{for all } f \in \cH. 
    \end{equation}
  \end{enumerate}
The constants $A$ and $B$ are called \emph{frame bounds}.
\end{definition}
\begin{remark}
  As we will only consider separable Hilbert spaces in this paper, we can replace weak measurability
  of $k\mapsto f_k$ with (strong) measurability with respect to the Borel algebra in $\cH$ by
  Pettis' theorem.
\end{remark}
In cases where it will cause no confusion, we will simply say that $\set{f_k}_{k \in M}$ is a frame
for $\cH$.  If $\set{f_k}_{k\in M}$ is weakly measurable and the upper bound in the above
inequality~(\ref{eq:cont-frame-inequality}) holds, then $\set{f_k}_{k\in M}$ is said to be a
\emph{Bessel family} with constant $B$. A frame $\set{f_k}_{k\in M}$ is said to be \emph{tight} if
we can choose $A = B$; if, furthermore, $A = B = 1$, then $\set{f_k}_{k\in M}$ is said to be a
\emph{Parseval frame}. 

Two Bessel families $\set{f_k}_{k\in M}$  and $\set{g_k}_{k\in M}$ are said to be \emph{dual frames} if
\begin{equation}
  \label{eq:cont-dual-weak}
  \innerprod{f}{g} = \int_M \innerprod{f}{g_k} \innerprod{f_k}{g} d\cfm (k) \quad \text{for all } f,g \in \cH.
\end{equation}
In this case we say that the following assignment
\begin{equation}
 f = \int_M \innerprod{f}{g_k}f_k \, d\cfm (k) \quad \text{for $f \in \cH$}, \label{eq:frame-rep-weak-sense}
\end{equation}
holds in the weak sense. Equation~(\ref{eq:frame-rep-weak-sense}) is often called a
\emph{reproducing formula} for $f \in \cH$. The following argument shows that two such dual frames
indeed are frames, and we shall say that the frame $\set{f_k}_{k\in M}$ is dual to $\set{g_k}_{k\in
  M}$, and vice versa. We need to show that both Bessel families $\set{f_k}_{k\in M}$ and
$\set{g_k}_{k\in M}$ satisfy the lower frame bound. By taking $f=g$ in~(\ref{eq:cont-dual-weak}) and
using the Cauchy-Schwarz inequality, we have
 \begin{align*}
  \Vert f \Vert^2 & = \int_M \langle f, f_k\rangle\langle g_k,f\rangle \,
  d\cfm(k) \le \Big( \int_M |\langle f, f_k\rangle|^2 \, d\cfm(k) \Big)^{1/2}\Big( \int_M |\langle
  f, g_k\rangle|^2 \, d\cfm(k) \Big)^{1/2} \\ & \le \Big( \int_M |\langle f, f_k\rangle|^2 \,
  d\cfm(k) \Big)^{1/2} \sqrt{B_g} \ \Vert f \Vert .
\end{align*}
In the last step we used that $\set{g_k}_{k\in M}$ has an upper frame bound $B_g$. Rearranging the
terms in the above inequality gives
\[ 
\frac{1}{B_g} \, \Vert f \Vert^2 \le \int_M |\langle f, f_k\rangle|^2 \, d\cfm(k) .
\] 
Hence, the Bessel family $\set{f_k}_{k\in M}$ satisfies the lower frame condition and is a
frame. 
A similar argument
shows that $\set{g_k}_{k\in M}$ satisfies the lower frame condition. 
This completes the argument. Moreover, by a polarization argument, it follows that two Bessel
families $\{f_k\}_{k\in M}$ and $\{g_k\}_{k\in M}$ are dual frames if, and only if,
\begin{equation*} 
\langle f,f\rangle = \int_M \langle f, g_k\rangle
   \langle f_k, f\rangle \, d\mu_M(k) \quad \text{for all } f \in \mathcal H. 
\end{equation*}
We mention that to a given frame for $\cH$ one can always find at least one dual frame. For more
information on (continuous) frames, we refer to
\cite{MR1206084,MR2428338,MR2167169,MR1735075,MR1287849,MR1968116}.

To a frame $\set{f_k}_{k\in M}$ for $\cH$, we associate the \emph{frame transform} given by
\[ 
\mathcal{H} \to L^2(M,\cfm), \quad f \mapsto (k \mapsto \innerprod{f}{f_k}).
\] 
As mentioned in the introduction, this transform is isometric if, and only if,
the family $\set{f_k}_{k\in M}$ is a Parseval frame. A similar conclusion holds for a pair of dual frames.

Let $(M_1,\Sigma_1,\mu_1)$ and $(M_2,\Sigma_2,\mu_2)$ be measure spaces.  We say that a family
$\set{f_k}_{k\in M_1}$ in the Hilbert space $\mathcal{H}$ is \emph{unitarily equivalent} to a family
$\set{g_k}_{k\in M_2}$ in the Hilbert space $\mathcal{K}$ if there is a point isomorphism $\iota\colon
M_1 \to M_2$, \ie $\iota$ is a (measurable) bijection such that $\iota(\Sigma_1)=\Sigma_2$ and $\mu_1 \circ
\iota^{-1}=\mu_2$, a unitary mapping $U\colon\mathcal{K}\to \mathcal{H}$, and measurable mapping $M_1
\to \C, k \mapsto c_k$ with $\abs{c_k}=1$ such that $f_k = c_k U g_{\iota(k)}$ for all $k \in
M_1$. This notion of unitarily equivalence generalizes a similar concept from \cite{MR1206084}.
Unitarily equivalence is important to us since it preserves many of the properties we are
interested in, \eg the frame property, including the frame bounds. The following lemma tells us that
``pairwise'' unitarily equivalence preserves the property of being dual frames.

\begin{lemma}
\label{thm:unitary-equi-duals}
Let $\set{f_k}_{k\in M_1}$ and $\setsmall{\tilde f_k}_{k\in M_1}$ be families in $\mathcal{H}$, and
let $\set{g_k}_{k\in M_2}$ and $\set{\tilde g_k}_{k\in M_2}$ be families in $\mathcal{K}$. Suppose
that
\[ 
f_k = c_k U g_{\iota(k)} \quad \text{and} \quad \tilde{f}_k = c_k U \tilde{g}_{\iota(k)}
\] 
for some point isomorphism $\iota\colon M_1 \to M_2$, a unitary mapping $U\colon\mathcal{K}\to
\mathcal{H}$, and a measurable mapping $M_1 \to \C, k \mapsto c_k$ with $\abs{c_k}=1$ for $k \in
M_1$. Then $\set{f_k}_{k \in M_1}$ and $\setsmall{\tilde{f}_k}_{k \in M_1}$ are dual frames with
respect to $(M_1,\Sigma_1,\mu_1)$ if, and only if, $\set{g_k}_{k \in M_2}$ and $\set{\tilde{g}_k}_{k
  \in M_2}$ are dual frames with respect to $(M_2,\Sigma_2,\mu_2)$.
\end{lemma}

\begin{proof}
  Assume that $\set{f_k}_{k\in M_1}$ and $\setsmall{\tilde f_k}_{k\in M_1}$ are a pair of dual
  frames. Since the composition of measurable functions is again measurable, then by our
  assumptions it follows that $\set{g_k}_{k\in M_2}$ and $\set{\tilde g_k}_{k\in M_2}$ are weakly
  measurable. They are obviously Bessel families. For $f \in \mathcal{K}$ and $g \in \cH$ we compute: 
  \begin{align*} 
    \innerprod{f}{U^\ast g} &= \innerprod{Uf}{g}=\int_{M_1} \langle Uf,\tilde{f}_k\rangle \innerprod{f_k}{g}
    d\mu_1(k) = \int_{M_1} \innerprod{Uf}{ c_k U \tilde g_{\iota(k)}}\innerprod{c_k U g_{\iota(k)}}{g}
    d\mu_1 (k) \\ & = \int_{M_1}  \innerprod{f}{\tilde g_{\iota(k)}}\innerprod{ g_{\iota(k)}}{U^\ast g}
    d\mu_1 (k) = \int_{M_2}  \innerprod{f}{\tilde g_{k}}\innerprod{ g_{k}}{U^\ast g}
    d\mu_2 (k),
  \end{align*}
  where the last equality follows from the properties of the point isomorphism.  Since $U^\ast$ is
  invertible on all of $\mathcal{K}$, this implies that $\set{g_k}_{k \in M_2}$ and $\set{\tilde{g}_k}_{k \in
    M_2}$ are dual frames. The opposite implication follows by symmetry.
\end{proof}

If $\cfm$ is the counting measure and $\Sigma_M=2^M$ the discrete $\sigma$-algebra, we say
that $\set{f_k}_{k\in M}$ is a \emph{discrete frame} whenever
(\ref{eq:cont-frame-inequality}) is satisfied; for this measure space, any family of vectors is
obviously weakly measurable. For discrete frames, equation~(\ref{eq:frame-rep-weak-sense}) holds in the
usual strong sense, \ie with (unconditional) convergence in the $\cH$ norm.

Lastly, we combine the notion of continuous frames with that of generalized frames, also known as
g-frames. Let $(M_j,\Sigma_j,\mu_j)$ be a measure space for each $j \in J$, where $J \subset \Z$ is
a countable index set. We will say that a union $\cup_{j \in J}\set{f_{j,k}}_{{k} \in M_j}$ is a
g-frame for $\cH$, or simply a frame, with respect to $\setpropsmall{L^2(M_j,\mu_j)}{j \in J}$ if
  \begin{enumerate}[(a)]
  \item $k \mapsto f_{j,k}, M_j \to \cH$ is measurable for each $j \in J$,  and
  \item there exist constants $A,B>0$ such that
    \begin{equation}
      \label{eq:cont-g-frame-inequality} 
       A \norm{f}^2 \le \sum_{j \in J}\int_{M_j} \abs{\innerprod{f}{f_{j,k}}}^2 d\mu_{M_j}({k}) 
      \le B \norm{f}^2 \quad \text{for all } f \in \cH. 
    \end{equation}
  \end{enumerate}
  The above definition and statements about continuous frames carry over to continuous g-frames; we
  refer to the original paper by Sun~\cite{MR2239250} for a detailed account of
  g-frames. Lemma~\ref{thm:unitary-equi-duals} is also easily transferred to this new setup. We will
  repeatedly use that it is sufficient to verify the various frame properties on a dense subset of
  $\cH$. The precise statement is as follows.
  \begin{lemma}
    \label{lem:dual-frames-on-dense-subset}
    Let $\cD$ be a dense subset of $\cH$, and let $(M_j,\mu_j)$ be a measure space for each $j \in J$.
    \begin{enumerate}[(i)]
    \item Suppose that $\cup_{j \in J}\set{f_{j,k}}_{k \in M_j}$ and $\cup_{j \in J}\set{g_{j,k}}_{k \in M_j}$ are Bessel families in $\cH$. If, for  $f \in \cD$,
   \begin{equation}
      \label{eq:cont-g-dual-frames-cond} 
       \innerprod{f}{f} = \sum_{j \in J}\int_{M_j} \innerprod{f}{f_{j,k}} \innerprod{g_{j,k}}{f} \, d\mu_{M_j}({k}), 
    \end{equation}
    then equation~(\ref{eq:cont-g-dual-frames-cond}) holds for all $f \in \cH$, \ie $\cup_{j \in
      J}\set{f_{j,k}}_{k \in M_j}$ and $\cup_{j \in J}\set{g_{j,k}}_{k \in M_j}$ are dual frames. \label{enum:dense-frame-dual}
  \item Suppose that $(M_j,\mu_{M_j})$ are $\sigma$-finite and $\cup_{j \in J}\set{f_{j,k}}_{k \in M_j}$ weakly measurable. If, for  $f \in \cD$,
   \begin{equation}
      \label{eq:cont-g-Parseval-frame-cond} 
       \innerprod{f}{f} = \sum_{j \in J}\int_{M_j} \innerprod{f}{f_{j,k}} \innerprod{f_{j,k}}{f} \, d\mu_{M_j}({k}), 
    \end{equation}
    then equation~(\ref{eq:cont-g-Parseval-frame-cond}) holds for all $f \in \cH$, \ie $\cup_{j \in
      J}\set{f_{j,k}}_{k \in M_j}$ is a Parseval frame. \label{enum:dense-frame-Parseval}
    \end{enumerate}
  \end{lemma}
  \begin{proof}
    (i): The first statement follows by a straightforward generalization of the proof of the same result for discrete frames~\cite[Lemma 7]{MR1600215}. The duality of $\cup_{j \in J}\set{f_{j,k}}_{k \in M_j}$ and $\cup_{j \in J}\set{g_{j,k}}_{k \in M_j}$ follows then by polarization. \\
    \noindent (ii): Without loss of generality we can assume that the measure space
    $(M_j,\mu_{M_j})$ is bounded for each $j \in J$. By use of Lebesgue's bounded convergence
    theorem, equation~(\ref{eq:cont-g-Parseval-frame-cond}) for $f \in \cD$ implies that $\cup_{j
      \in J}\set{f_{j,k}}_{k \in M_j}$ is a Bessel family on all of $\cH$; a similar argument can be
    found in the proof of \cite[Proposition 2.5]{MR2238038}. The result now follows from (i).
  \end{proof}

\section{Generalized translation invariant systems}
\label{sec:gsi}

In this section we will work with generalized translation invariant systems $\gti$, introduced in
Section~\ref{sec:gener-transl-invar}, in the setting of continuous g-frames. In order to do this, we
let $(P_j,\Sigma_{P_j},\mu_{P_j})$ be a $\sigma$-finite measure space for each $j \in J$, where $J
\subset \Z$ is a countable index set. For a topological space $T$, we let $B_T$ denote the Borel
algebra of $T$. We now consider $M_j:=P_j \times \LG_{\! j}$, and let $\Sigma_{M_j}:=\Sigma_{P_j}
\otimes B_{\LG_{\! j}}$ and $\mu_{M_j}:=\mu_{P_j} \otimes \mu_{\LG_{\! j}}$ denote the product
algebra and the product measure on $P_j \times \LG_{\! j}$, respectively.

 We will work under the following \textbf{standing hypotheses} on the generalized
translation invariant system $\gti$. For each $j \in J$:
\begin{enumerate}[(I)]
\item $(P_j,\Sigma_{P_j},\mu_{P_j})$ is a $\sigma$-finite measure space, \label{eq:Hypo1}
\item the mapping $p \mapsto g_p, (P_j,\Sigma_{P_j}) \to (L^2(G),B_{L^2(G)})$ is
  measurable, \label{eq:Hypo2} 
\item the mapping $(p,x)\mapsto g_p(x), (P_j \times G, \Sigma_{P_j} \otimes B_G) \to (\C, B_\C)$ is
  measurable. \label{eq:Hypo3} 
\end{enumerate}

Consider $T_{\gamma} g_p$ as a function of $(p,\gamma)\in P_j \times \LG_{\! j}$ into $L^2(G)$. This
function is continuous in $\gamma$ and measurable in $p$. Such functions are sometimes called
Carath\'eodory functions, and since $\LG_{\! j} \subset G$ is a second countable metric space, it
follows that any Carath\'eodory function, in particular $T_{\gamma} g_p$, is jointly measurable on
$(M_j,\Sigma_{M_j})=(P_j \times \LG_{\! j}, \Sigma_{P_j} \otimes B_{\LG_{\! j}})$. Thus, the family
of functions $\gti$ is automatically weakly measurable. A generalized translation invariant system
is therefore a frame for $L^2(G)$ if (\ref{eq:cont-g-frame-inequality}) is satisfied with respect to
the measure spaces $(M_j,\Sigma_{M_j},\mu_{M_j})$. Similar conclusions are valid with respect to
generalized translation invariant systems being Bessel families, Parseval frames, \etc Let us here
just observe that for dual frames $\gti$ and $\gti[h]$, we have the reproducing formula
\begin{equation*}
   f = \sum_{j\in J} \int_{P_j}\int_{\LG_{\!j}} \innerprod{f}{T_\gamma g_p}T_\gamma h_p \, d\mu_{\LG_{\!j}} (\gamma) \, d\mu_{P_j}(p) \quad  \text{for $f \in L^2(G)$},
\end{equation*}
where the measure on $\LG_{\!j}$ is chosen so that the measure on $\LG_{\!j}^\perp$ is the counting measure.

\begin{remark}
\label{rem:stand-hypo-special-cases}
  In Section~\ref{sec:gsi} we always assume the three standing hypotheses. However, in many special
  cases these assumptions are automatically satisfied: 
\begin{enumerate}[(a)]
\item When $P_j$ is
countable for all $j\in J$, we will equip it with a scaled counting measure $k \mu_c$, $k>0$, and the discrete
$\sigma$-algebra $2^{P_j}$. If all $P_j$, $j \in J$, are
countable, all three standing hypotheses therefore trivially hold. 
\label{rem:P-countable}
\item If $P_j$ is a second countable metric space for all $j\in J$ and if $p \mapsto g_p$ is
  continuous, then the standing hypotheses \eqref{eq:Hypo2} and \eqref{eq:Hypo3} are
  satisfied. Hence, if $P_j$ is also a subset of $G$ or $\ghat$ equipped with their respective Haar
  measure, then all three standing hypotheses hold.  \label{rem:count-wrt-p}
\end{enumerate}
\end{remark}

The main characterization results are stated in Theorem~\ref{th:2001e}
and~\ref{th:parseval-char-LCA}. These results rely on the following technical assumption.

\begin{definition} \label{def:asc} We say that two generalized translation invariant systems \gti[g]
  and \gti[h] satisfy the \emph{dual $\alpha$ local integrability condition} (dual $\alpha$-LIC) if,
  for all $f\in\mathcal D$,
  \begin{equation} \label{eq:asc} \sum_{j\in J} \int_{P_j} \sum_{\alpha\in \LG_{\! j}^{\perp}}
    \int_{\ghat} \absbig{\hat{f}(\omega) {\hat{f}(\omega\alpha)} {\hat
        g_p(\omega)}\hat{h}_p(\omega\alpha)} \, d\mu_{\ghat} (\omega) \, d\mu_{P_j}(p) < \infty.
\end{equation}
In case $g_p=h_p$ we refer to \eqref{eq:asc} as the \emph{$\alpha$ local integrability condition}
($\alpha$-LIC) for the generalized translation invariant system \gti.
\end{definition}

The $\alpha$-LIC should be compared to the local integrability condition for generalized shift
invariant systems introduced in \cite{MR1916862} for $L^2(\R^n)$ and in \cite{MR2283810} for
$L^2(G)$. For generalized translation invariant systems \gti the \emph{local integrability
  conditions} (LIC) becomes 
\begin{equation}
\label{eq:LIC} 
\sum_{j\in J} \int_{P_j} \sum_{\alpha\in \LG_{\! j}^{\perp}} \int_{\text{supp}\, \hat{f}} \absbig{ \hat{f}(\omega\alpha) \hat{g}_p(\omega) }^2  d\mu_{\ghat}(\omega) \, d\mu_{P_j}(p) < \infty\quad \text{for all $f\in\mathcal D$}. 
\end{equation}
Since the integrands in (\ref{eq:asc}) and (\ref{eq:LIC}) are measurable on $\cP_j \times \ghat$, we
are allowed to reorder sums and integrals in the local integrability conditions.

We will see (Lemma~\ref{le:LIC} and Example~\ref{ex:0402e}) that the LIC implies the $\alpha$-LIC,
but not vice versa.  Moreover, we mention that \emph{dual} local integrability conditions have not
been considered in the literature before. The following simple observation will often be used.

\begin{lemma} \label{le:ASC-compact-sets-and-functions} 
The following assertions are equivalent:
\begin{enumerate}[(i)]
\item The systems \gti and \gti[h] satisfy the dual $\alpha$-LIC, 
\item for each compact subset $K \subseteq \ghat \setminus E$ 
\[ \displaystyle \sum_{j\in J} \int_{P_j} \sum_{\alpha\in \LG_{\! j}^{\perp}}
  \int_{K\cap\alpha^{-1}K} \absbig{\hat{g}_p(\omega) \hat{h}_p(\omega\alpha)} \,
  d\mu_{\ghat} (\omega) \, d\mu_{P_j}(p) < \infty. \]
\end{enumerate}
\end{lemma}

\begin{proof} 
To show that (i) implies (ii), let $K$ be any compact subset in $\ghat$ and define
$\hat{f} = \mathds{1}_{K}$. Then, by assumption,
\begin{multline*}
\sum_{j\in J} \int_{P_j} \sum_{\alpha\in \LG_{\! j}^{\perp}} \int_{\ghat} \absbig{\hat{f}(\omega)
\hat{f}(\omega\alpha)\hat{g}_p(\omega)\hat{h}_p(\omega\alpha)}\,
d\mu_{\ghat} (\omega) \, d\mu_{P_j}(p) \\
 = \sum_{j\in J} \int_{P_j} \sum_{\alpha\in \LG_{\! j}^{\perp}}
\int_{K\cap\alpha^{-1}K} \absbig{\hat{g}_p(\omega)\hat{h}_p(\omega\alpha)} \, d\mu_{\ghat}
(\omega) \, d\mu_{P_j}(p)< \infty.
\end{multline*}
 To show that (ii) implies (i), take $f\in\mathcal D$ and denote $\text{supp}\,\hat{f}$ by
 $K$. Note that $\hat{f} \in L^{\infty}(\ghat)$.  Hence, we find that
  \begin{multline*}
\sum_{j\in J} \int_{P_j} \sum_{\alpha\in \LG_{\! j}^{\perp}} \int_{\ghat} \absbig{ \hat{f}(\omega)
\hat{f}(\omega\alpha) \hat{g}_p(\omega)\hat{h}_p(\omega\alpha)}\,
d\mu_{\ghat} (\omega) \, d\mu_{P_j}(p) \\ 
 \le \Vert \hat{f} \Vert_{\infty}^2 \ \sum_{j\in J} \int_{P_j} \sum_{\alpha\in
  \LG_{\! j}^{\perp}} \int_{K\cap\alpha^{-1}K} \absbig{ \hat{g}_p(\omega) \hat{h}_p(\omega\alpha) } \,
d\mu_{\ghat} (\omega) \, d\mu_{P_j}(p) < \infty.
\end{multline*}
\end{proof}

In a similar way, we see that \gti satisfies the local integrability condition if, and only if, for
each compact subset $K\subseteq \ghat\setminus E$
\begin{equation}
\label{eq:LIC-with-K}
\sum_{j\in J} \int_{P_j} \sum_{\alpha\in \LG_{\! j}^{\perp}}
\int_{K\cap\alpha^{-1}K} \absbig{ \hat{g}_p(\omega) }^2 \, d\mu_{\ghat}
(\omega) \, d\mu_{P_j}(p) < \infty.
\end{equation}

Inspired by the definition of the Calder\'on sum in wavelet theory, we will say that the term
$\sum_{j\in J} \int_{P_j} \abs{\hat{g}_p(\omega)}^2 d\mu_{P_j}(p)$ is the \emph{Calder\'on
  integral}.  The next result shows that the Calder\'on integral is bounded if the generalized
translation invariant system is a Bessel family. From this it follows that the
$t_{\alpha}$-equations~(\ref{eq:1702d}) are well-defined.  We remark that
Proposition~\ref{pr:bessel-implies-calderon-bound} generalizes \cite[Proposition~3.6]{MR2283810} and
\cite[Proposition 4.1]{MR1916862} from the uniform lattice setting where each $P_j$ is countable to
the setting of generalized translation invariant systems.

\begin{proposition} \label{pr:bessel-implies-calderon-bound} If the generalized translation
  invariant system \gti is a Bessel
  family with bound $B$, then 
  \begin{equation}
    \label{eq:calderon-bounded-by-bessel}
    \sum_{j\in J} \int_{P_j} \abs{\hat{g}_p(\omega)}^2 d\mu_{P_j}(p) \le B \quad \text{for a.e. }  \omega \in \ghat.
  \end{equation}
\end{proposition}

\begin{proof}
  We begin by noting that the Calder\'on integral in~(\ref{eq:calderon-bounded-by-bessel}) is
  well-defined by our standing hypothesis~\eqref{eq:Hypo3}. We assume without loss of generality
  that $J=\mathbb Z$. From the Bessel assumption on \gti, we have
  \[ \sum_{\abs{j} \le M} \int_{P_j} \int_{\LG_{\! j}} \abs{\innerprod{f}{T_{\gamma} g_p}}^2 \,
  d\mu_{\Gamma_{\!  j}}(\gamma) \, d\mu_{P_j}(p) \le B \norm{f}^2
  \]
  for every $M \in \N$ and all $f \in L^2(G)$. By Lemma \ref{le:1902a} we then get
  \begin{equation}
\sum_{\abs{j}\le M} \int_{P_j} \sum_{\alpha \in \LG_{\!  j}^\perp} \int_{\ghat} \hat{f}
  (\omega) \overline{\hat{f} (\omega \alpha)} \overline{\hat{g}_p (\omega)} \hat{g}_p (\omega
   \alpha) \,  d\mu_{\ghat} (\omega) \, d\mu_{P_j}(p) \le B \norm{f}^2 \label{eq:bessel-cond-implies}
 \end{equation}
  for every $M \in \N$ and all $f \in \mathcal D$. Assume towards a contradiction that there exists
  a Borel subset $N \subset \ghat$ of positive measure $\mu_{\ghat}(N)>0$ for which
  \begin{equation*}
  \sum_{j\in J} \int_{P_j} \abs{\hat{g}_p(\omega)}^2 \, d\mu_{P_j}(p) > B \quad \text{for a.e. } \omega \in N.
  \end{equation*} 
  In \cite{MR2283810} it is assumed that $N$ contains an open ball, but this needs not be the
  case. However, since $\ghat $ is $\sigma$-compact, there exists a compact set $K$ so that
  $\mu_{\ghat}(K \cap N)>0$. Set $\delta_M:=\inf\setpropsmall{d(\alpha,1)}{\alpha \in \LG_{\!
      j}^\perp \setminus \{1\}, \abs{ j} \le M}$. For any discrete subgroup $\LG$ there exists a
  $\delta>0$ such that $B(x,\delta)\cap \LG=\set{x}$ for $x \in \LG$, where $B(x,\delta)$ denotes
  the open ball of radius $\delta$ and center $x$. It follows that $\delta_M>0$ since $\delta_M$ is
  the smallest of such radii about $x=1$ from a \emph{finite} union of discrete subgroups $\LG_{\!
    j}^{\perp}$. Let $\mathcal{O}$ be an open covering of $K$ of sets with diameter
  strictly less than $\delta_M/2$. 

  Since a finite subset of $\mathcal{O}$ covers $K$, there is an open set $B \in \mathcal{O}$ so
  that $\mu_{\ghat}(B \cap K \cap N)>0$. Define $f \in L^2(G)$ by
  \[
  \hat{f} = \mathds{1}_{B \cap K \cap N}.
  \] 
  By Remark~\ref{rem:E-can-be-empty} below, we can assume that $\oline{E}$ does not intersect the
  closure of $B \cap K \cap N$. Therefore, $f \in \mathcal D$ and by our
  assumption we have
  \begin{multline*}
   \sum_{\abs{j}\le M} \int_{P_j} \sum_{\alpha \in \LG_{\!  j}^\perp}
    \int_{\ghat} \hat{f} (\omega) \overline{\hat{f} (\omega \alpha)} \overline{\hat{g}_p (\omega) }
    \hat{g}_p (\omega
    \alpha) \, d\mu_{\ghat} (\omega) \, d\mu_{P_j}(p) \\
    = \int_{\ghat} \vert \hat{f}(\omega)\vert^2 \sum_{\abs{j}\le M}\int_{P_j}
    \abs{\hat{g}_p(\omega)}^2 \, d\mu_{P_j}(p) \, d\mu_{\ghat}(\omega),
  \end{multline*}
  where the change of the order of integration above is justified by an application of the
  Fubini-Tonelli theorem together with the Bessel assumption (\ref{eq:calderon-bounded-by-bessel})
  and our standing hypotheses \eqref{eq:Hypo1} and \eqref{eq:Hypo3}. By letting $M$ tend to infinity,
  we see that
  \[
  \sum_{j \in J} \int_{P_j} \sum_{\alpha \in \LG_{\!  j}^\perp} \int_{\ghat} \hat{f} (\omega)
  \overline{\hat{f} (\omega \alpha)} \overline{\hat{g}_p (\omega) } \hat{g}_p (\omega \alpha) \,
  d\mu_{\ghat} (\omega) \, d\mu_{P_j}(p) > B \norm{f}^2,
  \]
  which contradicts (\ref{eq:bessel-cond-implies}).
\end{proof}

\begin{remark}
  \label{rem:E-can-be-empty} 
  In case $\oline{E}$ intersects the closure of $A:= B \cap K \cap N$ in the proof of
  Proposition~\ref{pr:bessel-implies-calderon-bound}, one needs to approximate the function $f$ with
  functions from $\cD$ as defined in (\ref{eq:def-D}). As we will use such arguments several times
  in the remainder of this paper, let us consider how to do such a modification in this specific
  case. Define $E_A=\oline{E}\cap \oline{A}$ and
  \[
F_n =\setpropbig{\omega \in A}{\inf\setprop{d(\omega,a)}{a\in E_A}
    <\tfrac{1}{n}}, \quad \text{for each $n \in \N$}.
\] 
Define $\hat{f}_n=\mathds{1}_{A \setminus F_n}\in \mathcal D$.  Since
$F_{n+1} \subset F_{n}$ and $\mu_{\ghat}(F_1)<\infty$, we have
\begin{align*}
  \normsmall{\hat{f}-\hat{f}_n} = \mu_{\ghat}(F_n) \to \mu_{\ghat}(\cap_{n \in \N} F_n) = \mu_{\ghat}(E_A) = 0 \qquad \text{as $n \to \infty$},
\end{align*}
where $\hat{f} = \mathds{1}_{B \cap K \cap N}$. Finally, we use $\hat{f}_n$ in place of
$\hat{f}$ in the final argument of the proof above, and let $n \to \infty$.
\end{remark}

\subsection{Characterization results for dual and Parseval frames}
\label{sec:dual-frames}
\label{sec:parseval-frames}

We are ready to prove the first of our main results, Theorem~\ref{th:2001e}. Under the technical
dual $\alpha$-LIC assumption we characterize dual generalized translation invariant frames in terms
of $t_{\alpha}$-equations. We stress that these GTI systems are dual frames \emph{with respect to}
$\setpropsmall{L^2(M_j,\mu_j)}{j \in J}$ defined in the previous section. Recall that we assume a
Haar measure on $G$ to be given, and that we equip every $\LG_{\! j}^{\perp}\subset \ghat$ with the
counting measure.

\begin{theorem} \label{th:2001e} Suppose that \gti and \gti[h] are Bessel families satisfying the
  dual $\alpha$-LIC. Then the following statements are equivalent:
\begin{enumerate}[(i)]
 \item \gti and \gti[h] are dual frames for $L^2(G)$,
 \item for each $\alpha \in \bigcup_{j\in J} \LG_{\! j}^{\perp}$ we have
\begin{equation}
 \label{eq:1702d} t_{\alpha}(\omega) := \sum_{j\in J \, : \, \alpha\in\LG_{\! j}^{\perp}} \int_{P_j}
 \overline{\hat{g}_p(\omega)} \hat{h}_p(\omega\alpha ) \, d\mu_{P_j}(p) =
\delta_{\alpha,1} \quad \almoste\ \omega \in \ghat. 
\end{equation}
\end{enumerate}
\end{theorem}
\begin{proof}
  Let us first show that the $t_{\alpha}$-equations are well-defined.  Take $B$ to be a common
  Bessel bound for the two GTI families. By two applications of the Cauchy-Schwarz inequality and
  Proposition \ref{pr:bessel-implies-calderon-bound}, we find that
  \begin{align*}
    \sum_{j\in J \, : \, \alpha \in \Gamma_{\! j}^{\perp}} & \int_{P_j} \vert \hat{g}_p (\omega)\vert \vert \hat{h}_p(\omega\alpha)\vert \, d\mu_{P_j}(p) \le \sum_{j\in J} \int_{P_j}
    \vert \hat{g}_p (\omega) \vert \vert \hat{h}_p(\omega\alpha) \vert \, d\mu_{P_j}(p) \\ & \le \sum_{j\in J} \Big(
    \int_{P_j} \vert \hat{g}_p(\omega)\vert^2 \, d\mu_{P_j}(p)\Big)^{1/2}\Big( \int_{P_j} \vert \hat{h}_p(\omega\alpha) \vert^2 \, d\mu_{P_j}(p)\Big)^{1/2} \\ & \le \Big( \sum_{j\in J}  \int_{P_j} \vert \hat{g}_p(\omega) \vert^2 \, d\mu_{P_j}(p)\Big)^{1/2}\Big( \sum_{j\in J} \int_{P_j} \vert \hat{h}_p(\omega\alpha) \vert^2 \, d\mu_{P_j}(p)\Big)^{1/2}
      \le B,  
  \end{align*}
 for $\almoste \ \omega \in \ghat$. This shows that the $t_{\alpha}$-equations are well-defined and converge absolutely.

  For $f\in\mathcal D$, define the function
  \begin{equation}
  w_f:G\to \C, \ w_f(x) := \sum_{j\in J} \int_{P_j}\int_{\LG_{\! j}} \langle T_x f,T_{\gamma}
  g_p\rangle \langle T_{\gamma} h_p, T_xf \rangle \, d\mu_{\Gamma_{\!  j}}(\gamma) \, d\mu_{P_j}(p).\label{eq:wf-definition}
  \end{equation}
  By Lemma \ref{le:1902a} and the calculation $\widehat{T_xf}(\omega)
  \overline{\widehat{T_xf}(\omega\alpha)} = \alpha(x) \hat{f}(\omega) \overline{\hat{f}
    (\omega\alpha)} $, we have
  \begin{align*}
    w_f(x) = \sum_{j\in J} \int_{P_j} \int_{\ghat} \sum_{\alpha\in \LG_{\! j}^{\perp}} \alpha(x)
    \hat{f}(\omega) \overline{\hat{f}(\omega\alpha)} \overline{\hat{g}_p(\omega)}
    \hat{h}_p(\omega\alpha) \, d\mu_{\ghat}(\omega) \, d\mu_{P_j}(p).
  \end{align*}
  Let $\varphi_{\alpha,j}(p,\omega)$ denote the innermost summand in the right hand side expression
  above. By our standing hypothesis~\eqref{eq:Hypo3}, the function $\varphi_{\alpha,j}$ is
  $(\Sigma_{P_j} \otimes B_G)$-measurable for each $\alpha$. Applying  Beppo Levi's theorem  to the dual $\alpha$ local
  integrability condition yields that the function $\sum_\alpha \varphi_{\alpha,j}$ belongs to 
  $L^1(P_j \times \ghat)$ for each $j\in J$. An application of Fubini's theorem now gives:
  \begin{align*}
    w_f(x) = \sum_{j\in J} \int_{\ghat} \int_{P_j} \sum_{\alpha\in \bigcup\limits_{j \in J} \LG_{\!
        j}^{\perp}} \mathds{1}_{\LG_{\! j}^{\perp}}(\alpha) \varphi_{\alpha,j} (p,\omega) \,
    d\mu_{P_j}(p) \, d\mu_{\ghat}(\omega).
  \end{align*}
 Lebesgue's dominated convergence theorem then yields:
  \begin{align*}
    w_f(x) = \sum_{j\in J} \sum_{\alpha\in \bigcup\limits_{j\in J} \LG_{\! j}^{\perp}} \alpha(x)
    \int_{\ghat} \int_{P_j} \mathds{1}_{\LG_{\! j}^{\perp}}(\alpha) \hat{f}(\omega)
    \overline{\hat{f}(\omega\alpha)} \overline{\hat{g}_p(\omega)} \hat{h}_p(\omega\alpha) \,
    d\mu_{P_j}(p) \, d\mu_{\ghat}(\omega).
  \end{align*}
  By the dual $\alpha$ local integrability condition the summand belongs to $\ell^1(J\times
  \cup_{j\in J} \Gamma_{\! j}^{\perp})$ and we can therefore interchange the order of
  summations. Further, by Lebesgue's bounded convergence theorem, we can interchange the sum over $j
  \in J$ and the integral over $\supp{\hat{f}}\subset \ghat$. Hence,
\begin{align*}
  w_f(x) = \sum_{\alpha\in \bigcup\limits_{j\in J} \LG_{\! j}^{\perp}} \alpha(x) \int_{\ghat}
  \hat{f}(\omega) \overline{\hat{f}(\omega\alpha) } \sum_{j\in J \, : \, \alpha\in\LG_{\! j}^{\perp}
  } \int_{P_j} \overline{\hat{g}_p(\omega)} \hat{h}_p(\omega\alpha) \, d\mu_{P_j}(p) \,
  d\mu_{\ghat}(\omega).
  \end{align*}
  Finally, we arrive at:
  \begin{equation}
    w_f(x) = \sum_{\alpha\in
      \bigcup\limits_{j \in J} \LG_{\! j}^{\perp}} \alpha(x) \hat w(\alpha),  \quad \text{where} \quad 
    \hat w(\alpha) :=
    \int_{\ghat} \hat{f}(\omega) \overline{\hat{f}(\omega\alpha) }  t_{\alpha}(\omega) \,
    d\mu_{\ghat}(\omega).
    \label{eq:2502a}
  \end{equation}
  From the previous calculations and the dual $\alpha$-LIC, it follows that the convergence in
  (\ref{eq:2502a}) is absolute.  By the Weierstrass M-test, we see that $w_f$ is the uniform limit
  of a generalized Fourier series and thus an almost periodic, continuous function.

  We start by showing the implication (ii)$\Rightarrow$(i). 
  Inserting \eqref{eq:1702d} into \eqref{eq:2502a} for $x=0$ yields
  \begin{align*}
    w_f(0) & = \sum_{j\in J} \int_{P_j}\int_{\Gamma_{\!  j}} \langle f,T_{\gamma} g_p\rangle \langle T_{\gamma} h_p, f \rangle \, d\mu_{\Gamma_{\!  j}}(\gamma) \, d\mu_{P_j}(p) \\
    & = \sum_{\alpha\in \bigcup\limits_{j \in J} \LG_{\! j}^{\perp}} \alpha(0) \int_{\ghat}
    \hat{f}(\omega) \overline{\hat{f}(\omega\alpha) } \delta_{\alpha,1} \, d\mu_{\ghat}(\omega) =
    \langle f , f \rangle,
  \end{align*}
  and (i) follows by Lemma~\ref{lem:dual-frames-on-dense-subset}\eqref{enum:dense-frame-dual}. 

  For the converse implication (i)$\Rightarrow$(ii), we have
  \[
  w_f(x) = \sum_{j \in J} \int_{P_j}\int_{\Gamma_{\! j}} \langle T_x f,T_{\gamma} g_p\rangle \langle T_{\gamma}
  h_p, T_xf \rangle \, d\mu_{\Gamma_{ j}}(\gamma) \, d\mu_{P_j}(p) = \Vert f \Vert^2
  \]
  for each $f \in \cD$.  Consider now the function $z(x) := w_f(x) - \Vert f \Vert^2$.  We have
  shown that $w_f$ is continuous and by construction $z$ is identical to the zero function.
  Additionally, since $w_f$ equals an absolute convergent, generalized Fourier series, also $z$ can
  be expressed as an absolute convergent generalized Fourier series $\displaystyle z(x) =
  \sum_{\alpha\in\bigcup_{j \in J} \LG_{\! j}^{\perp}} \alpha(x) \hat{z}(\alpha)$, with
  \[
  \hat{z}(\alpha) \ = \ \begin{cases} \ \ \displaystyle \int_{\ghat} \absbig{\hat{f}(\omega)}^2
    t_{1}(\omega) \,
    d\mu_{\ghat}(\omega)  - \norm{f}^2 & \text{for} \ \alpha=1, \\[10pt]
    \ \ \displaystyle \int_{\ghat} \hat{f}(\omega) \overline{\hat{f}(\omega\alpha)} 
    t_{\alpha}(\omega) \, d\mu_{\ghat}(\omega) & \text{for} \ \alpha\in\bigcup_{j \in J}
    \LG_{\! j}^{\perp} \setminus\{1\}. \end{cases} 
  \] 
  By the uniqueness theorem for generalized Fourier series \cite[Theorem 7.12]{MR0481915}, the
  function $z(x)$ is identical to zero if, and only if, $\hat{z} (\alpha)= 0$ for all $\alpha \in
  \cup_{j \in J} \LG_{\! j}^{\perp}$.

  In case $\alpha=1$ we have $\int_{\ghat} \abssmall{\hat{f}(\omega)}^2\bigl(t_{1}(\omega)-1\bigr)
  d\mu_{\ghat}(\omega)=0$ for $f \in \cD$. Hence, since $\cD$ is dense in $L^2(G)$, we conclude that
  $t_{1}(\omega) = 1$ for $\almoste\ \omega\in\ghat$. For $\alpha\in\cup_{j \in J}
  \LG_{\! j}^{\perp} \setminus\{1\}$, we have
  \begin{equation}
    \label{eq:2502b} \int_{\ghat} \hat{f}(\omega) \overline{\hat{f}(\omega\alpha)}
    t_{\alpha}(\omega) \, d\mu_{\ghat}(\omega) = 0.
  \end{equation}
  Define the multiplication operator $M_{\overline{t_{\alpha} }}:L^2(\ghat)\to L^2(\ghat)$ by $
  M_{\overline{t_{\alpha} }}\hat{f}(\omega) = \overline{t_{\alpha}(\omega)} \hat{f}(\omega)$. This
  linear operator is bounded since by Proposition \ref{pr:bessel-implies-calderon-bound}
  $t_{\alpha}(\omega)\in L^\infty(\ghat)$. We can now rewrite the left hand side of \eqref{eq:2502b}
  as an inner-product:
  \[ 
  \langle \hat{f}, M_{\overline{t_{\alpha} }} T_{\alpha^{-1}} \hat{f} \rangle_{L^2(\ghat)} =0,
  \] 
  where $f \in\cD$.  Since $\mathcal D$ is dense in the complex Hilbert space $L^2(G)$, this implies
  that $M_{\overline{t_{\alpha} }} T_{\alpha^{-1}} = 0$. After multiplication with $T_\alpha$ from
  the right, we have $M_{\overline{t_{\alpha} }} = 0$ and therefore $t_{\alpha} = 0$.
\end{proof}

From Theorem~\ref{th:2001e} we easily obtain the corresponding characterization for tight frames. We
state it for Parseval frames only as it is just a matter of scaling.

\begin{theorem} \label{th:parseval-char-LCA} Suppose that the generalized translation invariant
  system \gti sa\-tis\-fi\-es the $\alpha$ local integrability condition. Then the following
  assertions are equivalent:
\begin{enumerate}[(i)]
 \item $\gti$ is a Parseval frame for $L^2(G)$,
 \item for each $\alpha \in \bigcup_{j \in J} \LG_{\! j}^{\perp}$ we have
\begin{equation*}
t_{\alpha} := \sum_{j\in J \, : \, \alpha \in \Gamma_{\! j}^{\perp}} \int_{P_j}
 \overline{\hat{g}_p(\omega)} \hat{g}_p(\omega\alpha ) \, d\mu_{P_j}(p)=
\delta_{\alpha,1} \quad \almoste\ \omega \in \ghat. 
\end{equation*}
\end{enumerate}
\end{theorem}
\begin{proof} 
  We first remark that the integrals in (ii) indeed converge absolutely. This follows from two applications of the
  Cauchy-Schwarz' inequality (as in the proof of Theorem~\ref{th:2001e}), which gives:
  \[
 \sum_{j\in J \, : \, \alpha \in \Gamma_{\! j}^{\perp}} \int_{ P_j } \abs{\hat{g}_p
    (\omega)} \abs{\hat{g}_p(\omega\alpha)}  d\mu_{P_j}(p) \le 1.
  \]
  
  In view of Theorem \ref{th:2001e} we only have to argue that the assumption on the Bessel family
  can be omitted. If we assume (i), then clearly \gti is a Bessel family and (ii) follows from Theorem
  \ref{th:2001e}.

  Suppose that (ii) holds. Formula~(\ref{eq:2502a}) is still valid, where $w_f$ is defined as in~(\ref{eq:wf-definition}) with $h_p=g_p$.
  Setting $x=0$ in (\ref{eq:2502a}) yields 
  \[
  \Vert f \Vert^2 = \sum_{j\in J} \int_{P_j} \int_{\LG_{\! j}} \abs{\langle f, T_{\gamma} g_p\rangle }^2  d\mu_{\Gamma_{\! j}}(\gamma) \,
d\mu_{P_j}(p) \quad \text{for all $f \in \cD$.}
  \] 
  Finally, we
  conclude by Lemma~\ref{lem:dual-frames-on-dense-subset}\eqref{enum:dense-frame-Parseval} that \gti is a Parseval frame for $L^2(G)$.
\end{proof}

By virtue of Lemma \ref{thm:unitary-equi-duals}, we have the following extension of Theorem~\ref{th:2001e} and
  \ref{th:parseval-char-LCA}.
  \begin{corollary} \label{cor:extension} The characterization results in Theorem~\ref{th:2001e} and
    \ref{th:parseval-char-LCA} extend to systems that are unitarily equivalent to generalized
    translation invariant systems.
\end{corollary}

\subsection{On sufficient conditions and the local integrability conditions}
\label{sec:suff-cond-local}
Let us now turn to sufficient conditions for a generalized translation invariant system to be a
Bessel family or a frame. Proposition \ref{pr:suff-cond-for-A-and-B-bound} is a generalization of
the results in, e.g., \cite{MR2420867} and \cite{OleSaySon2014}, which state the corresponding
result for GSI systems in the euclidean space and locally compact abelian groups, respectively. The
result is as follows.
\begin{proposition} 
  \label{pr:suff-cond-for-A-and-B-bound} Consider the generalized translation invariant system \gti.
\begin{enumerate}[(i)]
\item If
\begin{equation} \label{eq:CC-condition} B:= \esssup_{\omega\in \ghat} \sum_{j\in J} \int_{P_j} \sum_{\alpha\in
\LG_{\! j}^{\perp}}  \absbig{ \hat{g}_p(\omega) \hat{g}_p(\omega\alpha)} \, d\mu_{P_j}(p) < \infty,
\end{equation}
then \gti is a Bessel family with bound
$B$. 
\item Furthermore, if also
\[ A:= \essinf_{\omega\in\ghat} \Big( \sum_{j\in J} \int_{P_j} \abs{\hat
g_p(\omega)}^2  d\mu_{P_j}(p) - \sum_{j\in J} \int_{P_j} \sum_{\alpha\in
\LG_{\! j}^{\perp}\setminus\{0\}}  \abs{\hat{g}_p(\omega) \hat{g}_p(\omega
\alpha) }  d\mu_{P_j}(p) \Big) > 0, \]
then \gti is a frame for $L^2(G)$ with bounds $A$ and $B$.
\end{enumerate}
\end{proposition}
\begin{proof}
With a few adaptations the result follows from the corresponding proofs in \cite{OleSaySon2014} and \cite{MR2420867}. 
\end{proof}
We refer to \eqref{eq:CC-condition} as the absolute CC-condition, or for short, CC-condition \cite{MR1814424}. 
Proposition \ref{pr:suff-cond-for-A-and-B-bound} is useful in applications as a mean to verify that
a given family indeed is Bessel, or even a frame. Moreover, in relation to the characterizing
results in Theorem \ref{th:2001e} and \ref{th:parseval-char-LCA}, the
CC-condition~(\ref{eq:CC-condition}) is sufficient for the $\alpha$-LIC to
hold. In contrast, we remark that, by Example~\ref{ex:0402e} in
Section~\ref{sec:exampl-local-integr}, the CC-condition does not imply the LIC.
\begin{lemma} \label{le:CC-implies-alphaLIC} If \gti and \gti[h] satisfy
\[\esssup_{\omega\in\ghat} \sum_{j\in J} \int_{P_j} \sum_{\alpha\in \Gamma_{\! j}^{\perp}} \big\vert \hat g_p(\omega) \hat h_p(\omega \alpha) \big\vert \, d\mu_{P_j}(p) < \infty \phantom{,}\]
and
\[\esssup_{\omega\in\ghat} \sum_{j\in J} \int_{P_j} \sum_{\alpha\in \Gamma_{\! j}^{\perp}} \big\vert \hat g_p(\omega \alpha) \hat h_p(\omega ) \big\vert \, d\mu_{P_j}(p) < \infty,\]
then the dual $\alpha$ local integrability condition is satisfied. Furthermore, 
if \gti sa\-tis\-fies the CC-condition~(\ref{eq:CC-condition}), then the $\alpha$ local integrability condition is satisfied.
\end{lemma}
\begin{proof}
By applications of Cauchy-Schwarz' inequality, we find
\begin{align*}
& \sum_{j\in J} \int_{P_j} \sum_{\alpha\in \Gamma_{\! j}^{\perp}} \int_{\ghat} \vert\hat f(\omega) \hat f(\omega\alpha) \hat g_p(\omega) \hat h_p(\omega \alpha) \vert \, d\mu_{\ghat}(\omega) \, d\mu_{P_j}(p) \\
& \le \bigg[ \sum_{j\in J} \int_{P_j}  \sum_{\alpha\in \Gamma_{\! j}^{\perp}} \int_{\ghat} \vert\hat f(\omega) \vert^2 \vert \hat g_p(\omega) \hat h_p(\omega \alpha) \vert \, d\mu_{\ghat}(\omega) \, d\mu_{P_j}(p) \bigg]^{1/2} \\
& \qquad \qquad \qquad \qquad  \times \ \bigg[ \sum_{j\in J} \int_{P_j} \sum_{\alpha\in \Gamma_{\! j}^{\perp}} \int_{\ghat} \vert\hat f(\omega\alpha) \vert^2 \vert \hat g_p(\omega) \hat h_p(\omega \alpha) \vert \, d\mu_{\ghat}(\omega) \, d\mu_{P_j}(p) \bigg]^{1/2} \\
& = \bigg[ \int_{\ghat} \vert\hat f(\omega) \vert^2 \sum_{j\in J} \int_{P_j} \sum_{\alpha\in \Gamma_{\! j}^{\perp}} \vert \hat g_p(\omega) \hat h_p(\omega \alpha) \vert \, d\mu_{P_j}(p) \, d\mu_{\ghat}(\omega) \bigg]^{1/2} \\
& \qquad \qquad \qquad \qquad \times \ \bigg[ \int_{\ghat} \vert\hat f(\omega) \vert^2 \sum_{j\in J} \int_{P_j} \sum_{\alpha\in \Gamma_{\! j}^{\perp}} \vert \hat g_p(\omega\alpha) \hat h_p(\omega) \vert \, d\mu_{P_j}(p) \, d\mu_{\ghat}(\omega) \bigg]^{1/2} < \infty. 
\end{align*}
\end{proof}

Finally, we show that the LIC implies the (dual) $\alpha$-LIC. The precise statement is as follows.
\begin{lemma} \label{le:LIC} 
If both \gti and \gti[h] satisfy the local integrability condition \eqref{eq:LIC},
then \gti and \gti[h] satisfy the \emph{dual} $\alpha$ local integrability condition. In particular, if \gti satisfies the local integrability condition, then it also satisfies the $\alpha$ local integrability condition. 
\end{lemma}
\begin{proof}
By use of Cauchy-Schwarz' inequality and $2\abs{cd}\le\abs{c}^2+\abs{d}^2$, we have
\begin{align*}
  & \sum_{j\in J} \int_{P_j} \sum_{\alpha\in \Gamma_{\!  j}^{\perp}} \int_{\ghat} \absbig{ \hat{f}(\omega) \hat  f(\omega\alpha)\hat{g}_p(\omega)\hat{h}_p(\omega\alpha) } \, d\mu_{\ghat}(\omega) \, d\mu_{P_j}(p) \\
  & \le \sum_{j\in J} \int_{P_j} \sum_{\alpha\in \Gamma_{\!  j}^{\perp}} \bigg( \int_{\alpha^{-1} \supp \hat{f}} \absbig{ \hat{f}(\omega) \hat{h}_p(\omega\alpha) }^2  d\mu_{\ghat}(\omega) \bigg)^{1/2} \bigg( \int_{\supp \hat{f}} \absbig{ \hat  f(\omega\alpha)\hat{g}_p(\omega)}^2  d\mu_{\ghat}(\omega) \bigg)^{1/2}   d\mu_{P_j}(p) \\
  & = \sum_{j\in J} \int_{P_j} \sum_{\alpha\in \Gamma_{\!  j}^{\perp}} \bigg( \int_{\supp \hat{f}} \absbig{ \hat{f}(\omega\alpha^{-1}) \hat{h}_p(\omega)}^2  d\mu_{\ghat}(\omega) \bigg)^{1/2} \bigg( \int_{\supp \hat{f}} \absbig{ \hat  f(\omega\alpha)\hat{g}_p(\omega)}^2  d\mu_{\ghat}(\omega) \bigg)^{1/2}   d\mu_{P_j}(p)\\
  & \le \tfrac{1}{2} \sum_{j\in J} \int_{P_j} \sum_{\alpha\in \Gamma_{\!  j}^{\perp}} \int_{\supp
    \hat{f}} \absbig{ \hat{f}(\omega\alpha^{-1}) \hat{h}_p(\omega) }^2 \,
  d\mu_{\ghat}(\omega)  \, d\mu_{P_j}(p) \\
  & \quad \quad \ + \tfrac{1}{2} \sum_{j\in J} \int_{P_j} \sum_{\alpha\in \Gamma_{\!  j}^{\perp}}
  \int_{\supp \hat{f}} \absbig{ \hat{f}(\omega\alpha) \hat{g}_p(\omega) }^2 \,
  d\mu_{\ghat}(\omega)  \, d\mu_{P_j}(p) < \infty,
\end{align*}
and the statements follow.
\end{proof}

The relationships between the various conditions considered above are summarized in the diagram
below. To simplify the presentation we do not consider dual frames. An arrow means that the
assumption at the tail of the arrow implies the assumption at the head. A crossed out arrow means
that one can find a counter example for that implication; clearly, implications to the left in the
top line are also not true in general.
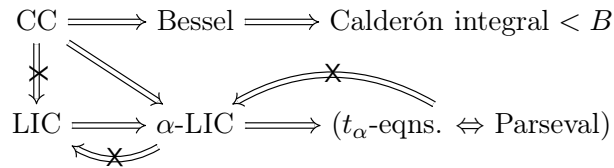
\[
\begin{tikzcd}
\text{CC} \arrow[Rightarrow,notarrow]{d} \arrow[Rightarrow]{r} \arrow[Rightarrow]{dr} &  \text{Bessel}  \arrow[Rightarrow]{r} & \text{Calder\'on integral} < B  \\
\text{LIC} \arrow[Rightarrow]{r} \arrow[Leftarrow,notarrow,out=-30,in=-150]{r}  &   \alpha\text{-LIC} \arrow[Rightarrow]{r} \arrow[Leftarrow,notarrow,out=30,in=150]{r} & \left(\text{$t_\alpha$-eqns. $\Leftrightarrow$ Parseval}\right)
\end{tikzcd}
\]
The crossed out arrows are shown by Example~\ref{ex:0402e} and Example~\ref{ex:talpha-not-satisfied} in the next section. 

\subsection{Two examples on the role of the local integrability conditions}
\label{sec:exampl-local-integr}

In this section we consider two key examples. Both examples take place in $\ell^2(\Z)$; however,
they can be extended to $L^2(\R)$, see \cite{MR2066821}. The first example, Example \ref{ex:0402e},
shows that for a GTI system the $\alpha$ local integrability condition is strictly weaker than the
local integrability condition.

\begin{example} 
  \label{ex:0402e}
  Let $G=\Z$, $N\in\N$, $N\ge 2$ and consider the co-compact subgroups $ \LG_{\! j} = N^j\Z,
  j\in\N.$ Note that $\ghat$ can be identified with the half-open unit interval $\itvco{0}{1}$ under
  addition modulo one. To each $\LG_{\! j}$ we associate $N^j$ functions $g_{j,p}$, for
  $p=0,1,\ldots,N^j-1$. Each function $g_{j,p}$ is defined by its Fourier transform
  \[
  \hat{g}_{j,p} = (N-1)^{1/2}N^{-j/2} \, \mathds{1}_{[p/N^j,(p+1)/N^j)}.
  \]
  The factor $(N-1)^{1/2}$ is for normalization purposes and does not play a role in the calculations. The annihilator of each $\Gamma_{\! j}$ is given by $\LG_{\!j}^{\perp}= N^{-j} \Z \cap [0,1)$. Note that the number of elements in $\LG_{\! j}^{\perp}$ is $N^j$. We equip both $G$ and $\Gamma_{\! j}^{\perp}$ with the counting measure, this implies that the measure on $\Gamma_{\! j}$ is the counting measure multiplied by $N^j$. For the generalized translation invariant system $\cup_{j\in \mathbb N} \{T_{\gamma}g_{j,p}\}_{\gamma\in \Gamma_{\! j},p=0,1,\ldots,N^{j}-1}$ we show the following: (i) the LIC is violated, (ii) the $\alpha$-LIC holds, (iii) the system is a Parseval frame for $\ell^2(\Z)$. It then follows from Theorem~\ref{th:parseval-char-LCA} that the $t_{\alpha}$-equations are satisfied. \\
  Ad (i). In order for the LIC to hold we need
  \[ \sum_{j=1}^{\infty} \sum_{p=0}^{N^j-1} \sum_{\alpha\in \LG_{\! j}^{\perp}} \int_{K \cap
    (K-\alpha)} \abs{ \hat{g}_{j,p}(\omega)}^2 \, d\omega \ < \ \infty\] for all compact $K\subseteq
  [0,1)$, see
  Lemma~\ref{le:ASC-compact-sets-and-functions}. 
  In particular for $K=\ghat$, we find
\begin{align*}
 & \quad \ \sum_{j=1}^{\infty} \sum_{p=0}^{N^j-1} \sum_{\alpha\in \LG_{\! j}^{\perp}}
\int_0^1\abs{\hat{g}_{j,p}(\omega)}^2
\, d\omega  = (N-1) \sum_{j=1}^{\infty} \sum_{p=0}^{N^j-1} \sum_{\alpha\in \LG_{\! j}^{\perp}}
 N^{-2j} \\ & = (N-1) \sum_{j=1}^{\infty} \sum_{p=0}^{N^j-1} N^{-j} = (N-1) \sum_{j=1}^{\infty} 1 = \infty.
\end{align*}
Therefore, the local integrability condition is not satisfied. \\
\noindent Ad (ii). By Lemma \ref{le:ASC-compact-sets-and-functions} it suffices to show that  
\[ \sum_{j=1}^{\infty} \sum_{p=0}^{N^j-1} \sum_{\alpha\in\LG_{\! j}^{\perp}} \int\limits_{\ghat \cap (\ghat-\alpha)} \absbig{\hat{g}_{j,p}(\omega)  \hat{g}_{j,p}(\omega+\alpha)} \, d\omega  < \infty. \]
Due to the support of $\hat{g}_{j,p}$ we have $\abssmall{ \overline{\hat{g}_{j,p}(\omega)}  \hat{g}_{j,p}(\omega+\alpha)}=0$ for $\alpha\in \LG_{\! j}^{\perp} \setminus\{ 0\}$. We thus find that
\begin{align*}
  & \quad \ \sum_{j=1}^{\infty} \sum_{p=0}^{N^j-1} \sum_{\alpha\in\LG_{\! j}^{\perp}} \int_0^1 \abs{\hat{g}_{j,p}(\omega)  \hat{g}_{j,p}(\omega+\alpha)} d\omega = \sum_{j=1}^{\infty} \sum_{p=0}^{N^j-1} \int_0^1 \abs{\hat{g}_{j,p}(\omega)}^2 \, d\omega \\
  & = (N-1) \sum_{j=1}^{\infty} \sum_{p=0}^{N^j-1} N^{-2j} = (N-1) \sum_{j=1}^{\infty} N^{-j} = 1.
\end{align*}
\noindent Ad (iii). Note that $\sum_{p=0}^{N^j-1} \abssmall{\hat{g}_{j,p}(\omega)}^2=(N-1)N^{-j}\, \mathds{1}_{[0,1)}(\omega)$ for $\omega\in[0,1)$ and for all $j\in\N$. Using the frame bound estimates from Proposition \ref{pr:suff-cond-for-A-and-B-bound}, we have 
\begin{align*}
 B & = \esssup_{\omega\in [0,1)} \sum_{j=1}^{\infty} \sum_{p=0}^{N^j-1} \sum_{\alpha\in
\LG_{\! j}^{\perp}}  \abs{\hat{g}_{j,p}(\omega) \hat{g}_{j,p}(\omega+\alpha)} \\
& = \esssup_{\omega\in [0,1)} \sum_{j=1}^{\infty} \sum_{p=0}^{N^j-1}   \abs{\hat{g}_{j,p}(\omega)}^2  = \esssup_{\omega\in [0,1)} (N-1) \sum_{j=1}^{\infty} N^{-j} \mathds{1}_{[0,1)}(\omega) = 1.
\end{align*}
In the same way, for the lower frame bound, we find
\[ 
A = \essinf_{\omega\in[0,1)} \Big( \sum_{j=1}^{\infty} \sum_{p=0}^{N^j-1} \abs{\hat
g_{j,p}(\omega)}^2 - \sum_{j=1}^{\infty} \sum_{p=0}^{N^j-1} \sum_{\alpha\in
\LG_{\! j}^{\perp}\setminus\{0\}}  \abs{\hat{g}_{j,p}(\omega) \hat{g}_{j,p}(\omega +
\alpha)}\Big) = 1. 
\]
These calculations also show that $\cup_{j\in\N} \{T_{\gamma} g_{j,p}\}_{\gamma\in \LG_{\! j},p=0,1,\ldots, N^j-1}$ is actually a union over $j \in\N$ of tight frames $\{T_{\gamma} g_{j,p}\}_{\gamma\in \LG_{\! j},p=0,1,\ldots, N^j-1}$ each with frame bound $N^{-j}$. Furthermore, we see that the CC-condition is satisfied, even though the LIC fails. Hence, the CC-condition does not imply LIC (however, by Lemma~\ref{le:CC-implies-alphaLIC} it does imply the $\alpha$-LIC). 
\end{example}

The following example is inspired by similar constructions in \cite{MR2066821} and \cite{MR2283810}. It shows two points. Firstly, the $\alpha$ local integrability condition cannot be removed in Theorem~\ref{th:parseval-char-LCA}. Secondly, 
it is possible for a GTI Parseval frame to satisfy the $t_\alpha$-equations even though the $\alpha$ local integrability condition fails. We show these observations in the reversed order.

\begin{example} \label{ex:talpha-not-satisfied}
  Let $G=\Z$ and for each $m \in \Z$ and $k \in \N$, let $[m]_k$ denote the residue class of $m$
  modulo $k$. Then, for $\tau_j=2^{j-1}-1$, $j \in \N$,
  \[
  \Z = \bigcup_{j \in \N} [\tau_j]_{2^j} = [0]_{2} \cup [1]_{4}\cup [3]_{8}\cup [7]_{16}\cup [15]_{32}
  \dots ,
  \]
  where the union is disjoint. Now set $g_j=N^{-j/2} \mathds{1}_{\tau_j}$ and $\LG_{\! j}=N^j\Z$ for $N=2$. The GTI system $\cup_{j\in\N} \set{T_\gamma g_j }_{\gamma \in \LG_{\! j}}$ is essentially a
  reordering of the standard orthonormal basis $\set{e_k}_{k \in \Z}$ for $\ell^2(\Z)$. The factor $N^{-j/2}$ in the definition of $g_j$ is due to the fact that we equip $\Gamma_{\! j}^{\perp}$ with the counting measure. This implies that the measure on $\Gamma_{\! j}$ becomes $N^j$ times the counting measure. One can now show that this GTI system does
  not satisfy the $\alpha$-LIC. However, the system does indeed satisfy the $t_{\alpha}$-equations. For $\alpha=0$:
  \[
  \sum_{j=1}^{\infty} \abs{\hat{g}_j(\omega)}^2=\sum_{j=1}^\infty 2^{-j} \abs{e^{2\pi i \tau_j
      \omega}}^2 = \frac{1}{2-1} = 1,
  \]
  and for $\alpha=k/2^{j^*} \in 2^{-j^*} \Z = \LG_{\! j^*}^\perp$, where $k$ is odd,
\begin{multline*}
  \sum_{j\in J \, : \, \alpha \in \Gamma_{\!j}^{\perp} } \hat{g}_j (\omega)
  \overline{\hat{g}_j (\omega + \alpha)} = \sum_{j=j^*}^\infty 2^{-j} e^{-2\pi i \frac{k}
    {2^{j^*}}(2^{j-1}-1)} = e^{2\pi i k2^{-j^*}} \sum_{j=j^*}^\infty 2^{-j} e^{-2\pi i \frac{2^{j-1}}{2^{j^*}}}
  \\ = e^{2 \pi i k2^{-j^*}} \bigg( -2^{-j^*} + \sum_{j=j^*+1}^\infty 2^{-j} \bigg) = 0.
\end{multline*}

If one uses $N\ge 3,N\in\N$ in place of $N=2$, then the $\alpha$-LIC is still not
satisfied. However, even though for suitably chosen $\tau_j$ (the formula is more complicated than
for $N=2$, see \cite{MR2066821}) $\cup_{j\in \N} \{T_{\gamma} N^{-j/2}
\mathds{1}_{\tau_j}\}_{\gamma\in N^j \Z}$ is still essentially a reordering of the standard
orthonormal basis, every $t_{\alpha}$-equation is false. The case $\alpha=0$ gives
$t_\alpha=\frac{1}{N-1}\neq 1$, while the cases $\alpha\neq 0$ give $t_\alpha \neq 0$. We stress
that these examples show the existence of generalized translation invariant Parseval frames for
$\ell^2(\Z)$ which do not satisfy the $t_{\alpha}$-equations.
\end{example}

\subsection{Characterization results for special groups}
\label{sec:special-syst}

Under special circumstances the local integrability condition will be satisfied automatically. In
this section we will see that this is indeed the case for TI systems, \ie $\Gamma_{\! j}=\Gamma$ for
all $j\in J$, and for GTI systems on compact abelian groups $G$. For brevity, we will only state the
corresponding characterization results for dual frames, but remark here that the results hold
equally for Parseval frames, in which case, the Bessel family assumption can be omitted.

Let us begin with a lemma concerning general GTI systems for LCA groups showing that the LIC holds
if the annihilators of $\Gamma_{\! j}$ possess a sufficient amount of separation.
\begin{lemma} \label{le:GSI-to-LIC} 
If \gti has a uniformly bounded Calder\'on integral and if there
  exists a constant $C>0$ such that for all compact $K\subseteq\ghat$
\begin{equation*}
\sum_{\alpha\in \bigcup\limits_{j\in J} \Gamma_{\! j}^{\perp}} \mu_{\ghat}(K\cap\alpha^{-1}K) \le C, 
\end{equation*} 
then \gti satisfies the local integrability condition. 
\end{lemma}
\begin{proof}
  By assumption there exists a constant $B>0$ such that $\sum_{j\in J}\int_{P_j}
  \abs{\hat{g}_p(\omega)}^2 \, d\mu_{P_j}(p) < B$ for $\almoste\ \omega\in\ghat$, and we therefore have
  \begin{align*}
& \sum_{j\in J} \int_{P_j} \sum_{\alpha \in \Gamma_{\! j}^{\perp}} \int_{K \cap \alpha^{-1} K} |\hat g_p(\omega)|^2 \, d\mu_{\ghat}(\omega) \, d\mu_{P_j}(p) \\
& = \sum_{\alpha\in \bigcup\limits_{j\in J} \Gamma_{\! j}^{\perp}} \int_{K\cap \alpha^{-1} K} \sum_{j\in J} \int_{P_j} |\hat g_p(\omega)|^2 \, d\mu_{\ghat}(\omega) \, d\mu_{P_j}(p) \le B C < \infty.
\end{align*}
\end{proof}

Now, let us consider the case where all subgroups $\Gamma_{\! j}$ coincide. In other words, we
consider \emph{translation invariant} systems. Note that this setting includes the continuous
wavelet and Gabor transform as well as the shift invariant systems considered in
\cite{MR1916862,MR2283810}.
\begin{theorem} \label{th:TI-systems} 
Let $\Gamma$ be a co-compact subgroup in $G$. Suppose that $\cup_{j\in J}
  \{T_{\gamma}g_p\}_{\gamma\in \Gamma, p\in P_j}$ and $\cup_{j\in J} \{T_{\gamma}h_p\}_{\gamma\in \Gamma, p\in
    P_j}$ are Bessel families. Then the following statements are equivalent:
\begin{enumerate}[(i)]
 \item $\cup_{j\in J} \{T_{\gamma}g_p\}_{\gamma\in \Gamma, p\in P_j}$ and $\cup_{j\in J} \{T_{\gamma}h_p\}
 _{\gamma\in \Gamma, p\in P_j}$ are dual frames for $L^2(G)$, 
 \item For each $\alpha \in \Gamma^{\perp}$ we have
   \begin{equation}
 t_{\alpha}(\omega) := \sum_{j\in J} \int_{P_j}
 \overline{\hat{g}_p(\omega)} \hat{h}_p(\omega\alpha ) \, d\mu_{P_j}(p) =
\delta_{\alpha,1} \quad \text{a.e.} \ \omega \in \ghat.\label{eq:t_alpha-SI}
\end{equation}
\end{enumerate}
\end{theorem}
\begin{proof}
  Since $\Gamma^{\perp}$ is a discrete subgroup in $\ghat$ and since the metric on $\ghat$ is
  translation invariant, there exists a $\delta>0$ so that the distance between two distinct points
  from $\Gamma^{\perp}$ is larger than $\delta$. Thus, for any compact $K \subset \ghat$, the set
  $\Gamma^\perp \cap (K^{-1}K)$ has finite cardinality because, if not, then $\Gamma^\perp \cap
  (K^{-1}K)$ would contain a sequence (take one without repetitions) with no convergent subsequence
  which contradicts the compactness of $K$. Since $\{\alpha\in\Gamma^{\perp} \, : \, K\alpha\cap
  K\ne \emptyset\}$ is a subset of $\Gamma^\perp \cap (K^{-1}K)$, it is also of finite cardinality.
  From this together with the Bessel assumption and Proposition
  \ref{pr:bessel-implies-calderon-bound} we conclude that the assumptions of Lemma
  \ref{le:GSI-to-LIC} are satisfied and hence the LIC holds. By Lemma \ref{le:LIC} the dual
  $\alpha$-LIC is satisfied and the result now readily follows from Theorem \ref{th:2001e}.
\end{proof}

For TI systems with translation along the entire group $\Gamma=G$ there is only one
$t_\alpha$-equation in (\ref{eq:t_alpha-SI}) since $G^\perp=\set{1}$. To be precise:
\begin{lemma} \label{th:TI-systems-along-G} Suppose that $\Gamma=G$. Then assertion (ii) in Theorem~\ref{th:TI-systems} reduces to  
\[ \sum_{j\in J} \int_{P_j}
 \overline{\hat{g}_p(\omega)} \hat{h}_p(\omega) \, d\mu_{P_j}(p) =
1 \quad \text{a.e.} \ \omega \in \ghat. 
\]
\end{lemma}

Let us now turn to the familiar setting of \cite{MR1916862,MR2283810}, where $\Gamma$ is a uniform
lattice, i.e., a discrete, co-compact subgroup. Then there is a compact fundamental domain $F\subset
G$ for $\Gamma$, such that $G=F\Gamma$, and moreover for any $x\in G$ we have $x = \varphi \gamma$,
where $\varphi\in F,\gamma\in \Gamma$ are unique.  For a uniform lattice we introduce the
\emph{lattice size} $s(\Gamma):=\mu_G(F)$, which is, in fact, independent of the choice of $F$.

\begin{corollary} \label{th:SI-uniform-lattice-dual} Let $\Gamma$ be a uniform lattice in $G$. Suppose that the two generalized translation invariant systems
  $\cup_{j\in J} \{T_{\gamma}g_p\}_{\gamma\in \Gamma, p\in P_j}$ and $\cup_{j\in J} \{T_{\gamma}h_p\}_{\gamma\in
    \Gamma, p\in P_j}$ are Bessel families. Then the following statements are equivalent:
\begin{enumerate}[(i)]
 \item $\cup_{j\in J} \{T_{\gamma}g_p\}_{\gamma\in \Gamma, p\in P_j}$ and $\cup_{j\in J}\{T_{\gamma}h_p\}
 _{\gamma\in \Gamma, p\in P_j}$ are dual frames for $L^2(G)$, \ie 
 \begin{equation}
  \langle f_1,f_2\rangle = \sum_{j\in J} \int_{P_j} s(\Gamma) \sum_{\gamma\in\Gamma} \langle f_1,T_{\gamma}g_p \rangle \langle T_{\gamma}h_p,f_2\rangle  \, d\mu_{P_j}(p), \quad \text{for all } f_1,f_2\in L^2(G).\label{eq:reprod-form-SI-dual-frames}
\end{equation}
 \item For each $\alpha \in \Gamma^{\perp}$ we have $t_\alpha(\omega)=\delta_{\alpha,1}$ for a.e. $\omega \in \ghat$, where $t_\alpha$ is defined in \eqref{eq:t_alpha-SI}
\end{enumerate}
\end{corollary}

\begin{remark} 
  In the same way, we can state the characterization results for \emph{generalized} shift-invariant
  systems. In this case we have countable many uniform lattices $\Gamma_{\! j}$, so we replace
  $s(\Gamma)$ in Corollary \ref{th:SI-uniform-lattice-dual} with $s(\Gamma_{\! j})$, sum over
  $\setpropsmall{j \in J}{\alpha \in \Gamma_{\! j}^\perp}$ in (\ref{eq:reprod-form-SI-dual-frames}),
  and add the dual $\alpha$ local integrability condition as assumption. We obtain a statement
  equivalent to the main characterization result in \cite{MR2283810}. In contrast to the result in
  \cite{MR2283810}, the lattice size $s(\Gamma)$ is contained in the reproducing formula rather than
  in the $t_{\alpha}$-equations.
\end{remark}

For compact abelian groups all generalized translation invariant systems satisfy the local
integrability condition. The characterization result is as follows.

\begin{theorem} \label{th:compact-G-dual} Let $G$ be a compact abelian group. Suppose that \gti and
  \gti[h] are Bessel families. Then the following statements are equivalent:
\begin{enumerate}[(i)]
 \item \gti and \gti[h] are dual frames for $L^2(G)$,
 \item for each $\alpha \in \bigcup\limits_{j\in J}\Gamma_{\! j}^{\perp}$ we have
\[ t_{\alpha}(\omega) := \sum_{j\in J \, : \, \alpha\in \LG_{\! j}^{\perp}} \int_{P_j}
 \overline{\hat{g}_p(\omega)} \hat{h}_p(\omega\alpha ) \, d\mu_{P_j}(p) =
\delta_{\alpha,1} \quad \text{a.e.} \ \omega \in \ghat. 
\]
\end{enumerate}
\end{theorem}
\begin{proof}
  Because $G$ is compact, the dual group $\ghat$ is discrete. All compact $K\subset \ghat$ are
  therefore finite. Let $\# K$ denote the number of elements in $K$. From the LIC we then find
\begin{align*}
\sum_{j\in J} \int_{P_j} \sum_{\alpha\in\Gamma_{\! j}^{\perp}} \sum_{\omega\in K\cap \alpha^{-1} K} \abs{\hat{g}_p(\omega)}^2 \, d\mu_{P_j}(p) & \le \sum_{j\in J} \int_{P_j} \# K \sum_{\omega\in K} \abs{\hat{g}_p(\omega)}^2 \, d\mu_{P_j}(p) \\ & \le (\# K)^2 \max_{\omega\in K} \, \sum_{j\in J} \int_{P_j} \abs{\hat{g}_p(\omega)}^2 \, d\mu_{P_j}.
\end{align*}
By the Bessel assumption and Proposition \ref{pr:bessel-implies-calderon-bound}, the Calder\'on
integral is bounded. The far right hand side in the above calculation is therefore finite, and the
LIC is satisfied. The result now follows from Theorem \ref{th:2001e} and Lemma \ref{le:LIC}.
\end{proof}

Finally, let us turn to discrete groups $G$. In this case, the local integrability condition is
\emph{not} automatically satisfied (as we saw in the examples in the previous section), but it has a
simple reformulation:
\begin{lemma} \label{le:discreteLIC} Suppose $G$ is a discrete abelian group. Then the following statements are equivalent:
\begin{enumerate}[(i)]
\item The system \gti satisfies the local integrability condition,
\item $\displaystyle \sum_{j\in J} \int_{P_j} \mu_c(\Gamma_{\! j}^{\perp}) \, \Vert g_p \Vert_{L^2(G)}^2 \, d\mu_{P_j}(p) < \infty$, where $\mu_c$ denotes the counting measure.
\end{enumerate}  
\end{lemma}
\begin{proof} Note that if $G$ is discrete, then $\ghat$ is compact. Hence the discrete groups $\Gamma_{\! j}^{\perp}$ are also compact and therefore finite. By this observation we can easily show the result. If (i) holds, then 
\[ \sum_{j\in J} \int_{P_j} \mu_c(\Gamma_{\! j}^{\perp}) \, \Vert g_p \Vert_{L^2(G)}^2 \, d\mu_{P_j}(p) \le \sum_{j\in J}\int_{P_j} \int_{\ghat} \mu_c( \Gamma_{\! j}^{\perp}) \, \vert\hat{g}_p(\omega) \vert^2 \, d\mu_{\ghat}(\omega) \, d\mu_{P_j}(p).\]
By \eqref{eq:LIC-with-K} with $K=\ghat$ the right hand side is finite, and (ii) follows. If (ii) holds, then
\[ \sum_{j\in J} \int_{P_j} \sum_{\alpha\in \LG_{\! j}^{\perp}}
\int_{K\cap\alpha^{-1}K} \vert \hat{g}_p(\omega) \vert^2  d\mu_{\ghat}
(\omega)  d\mu_{P_j}(p) \le \sum_{j\in J} \int_{P_j} \mu_c( \LG_{\! j}^{\perp})  
\int_{\ghat} \vert \hat{g}_p(\omega) \vert^2  d\mu_{\ghat}
(\omega)  d\mu_{P_j}(p) < \infty. \]
\end{proof}

\section{Applications and discussions of the characterization results}
\label{sec:applications}

In this section we study applications of Theorem~\ref{th:2001e} leading to new characterization
results. Moreover, we will easily recover known results as special cases of our theory. We consider
Gabor and wavelet-like systems for general locally compact abelian groups as well as for specific
locally compact abelian groups, \eg $\R^n, \Z^n, \Z_n$. We also give an example of characterization
results for the locally compact abelian group of $p$-adic numbers, where the theory of generalized
\emph{shift} invariant systems is not applicable. 

We will focus on verifying the local integrability conditions and on the deriving the characterizing
equations, but not on the related question of how to construct generators satisfying these
equations.  The recent work of Christensen and Goh \cite{OleSaySon2014} takes this more constructive
approach for generalized \emph{shift} invariant systems on locally compact abelian groups. Under
certain assumptions, they explicitly construct dual GSI frames using variants of
$t_\alpha$-equations, which are proved to be sufficient.

\subsection{Gabor systems}
\label{sec:gabor-systems}
A \emph{Gabor system} in $L^2(G)$ with generator $g\in L^2(G)$ is a family of functions of the form
\[\set{E_{\gamma} T_{\lambda} g}_{\gamma \in \LG, \lambda \in \LL}, \text{where $\LG\subseteq \ghat$ and
$\LL\subseteq
G$}.\] 
Note that a Gabor system $\set{E_{\gamma} T_{\lambda} g}_{\gamma \in \LG, \lambda \in \LL}$ is not a
generalized translation invariant system because $E_{\gamma} T_{\lambda} g = T_{\lambda} \big(
\gamma(\lambda) E_{\gamma} g \big)$ cannot be written as $T_{\gamma}{g_{j,p}}$ for $j \in J$ and $p \in P_j$
for any $\{g_{j,p}\}$.  However, by use of Lemma~\ref{thm:unitary-equi-duals}, we
can establish the following two possibilities to relate Gabor and translation invariant systems.

Firstly, by Lemma~\ref{thm:unitary-equi-duals} with $\iota=\mathrm{id}$, $U=\mathcal{F}$ and
$c_{\gamma,\lambda}=1$, we see that the Gabor system $\set{E_{\gamma} T_{\lambda} g}_{\gamma \in
  \LG, \lambda \in \LL}$ is a frame if, and only if, the translation invariant system $\{ T_{\gamma}
\mathcal F^{-1} T_{\lambda} g \}_{\gamma \in \LG, \lambda \in \LL}$ is a frame. By this observation
all results for translation invariant systems naturally carry over to Gabor systems. In order to apply the theory
established in this paper, we need $\Gamma$ to be a closed, co-compact subgroup of $\ghat$ and
$\Lambda$ to be equipped with a measure $\mu_{\Lambda}$ satisfying the standing
hypotheses~\eqref{eq:Hypo1}--\eqref{eq:Hypo3}. This approach together with Theorem \ref{th:2001e}
yield $t_{\alpha}$-equations in the \emph{time} domain $G$: for each $\alpha \in \LG^{\perp}$ we have
\[
  \int_{\Lambda}
 \overline{g(x-\lambda)} h(x-\lambda+\alpha ) \, d\mu_{\Lambda}(\lambda) =
\delta_{\alpha,0} \quad \almoste\ x \in G. 
\]

Secondly, by Lemma~\ref{thm:unitary-equi-duals} with $\iota=\mathrm{id}$, $U=\mathrm{id}$ and
$c_{\gamma,\lambda}=\gamma(\lambda)$, we see that the Gabor system $\set{E_{\gamma} T_{\lambda}
  g}_{\gamma \in \LG, \lambda \in \LL}$ is a frame if, and only if, the translation invariant system
$\set{T_{\lambda} E_{\gamma} g}_{\gamma\in\LG,\lambda\in \LL}$ is a frame. This time we need $\Lambda$ to be a closed,
co-compact subgroup of $G$ and $\Gamma$ to be equipped with a measure satisfying standing
hypotheses~\eqref{eq:Hypo1}--\eqref{eq:Hypo3}.  In contrast to the first approach, Theorem
\ref{th:2001e} now yields $t_{\alpha}$-equations in the \emph{frequency} domain $\ghat$: for each $\beta \in \Lambda^{\perp}$ we have
\[
  \int_{\Gamma}
 \overline{\hat{g}(\omega\gamma)} \hat{h}(\omega\gamma\beta ) \, d\mu_{\Gamma}(\gamma) =
\delta_{\beta,1} \quad \almoste\ \omega \in \ghat. 
\]

Gabor systems play a major role in time-frequency analysis \cite{gr01} and it is common to require
similar properties on $\Gamma$ and $\Lambda$. In the following theorem we characterize dual Gabor
frames, where we combine both of the above approaches and require that $\Lambda$ and $\Gamma$ are
closed, co-compact subgroups.  If we consider Parseval frames, then the Bessel assumption in Theorem
\ref{th:dual-Gabor} can be omitted.
\begin{theorem} 
\label{th:dual-Gabor} 
Let $\Lambda$ and $\Gamma$ be closed, co-compact subgroups of $G$ and $\ghat$ respectively and equip
$\Lambda^{\perp}$ and $\Gamma^{\perp}$ with the counting measure. Suppose that the two systems
$\{E_{\gamma}T_{\lambda}g\}_{\gamma\in\Gamma, \lambda\in\Lambda}$ and
$\{E_{\gamma}T_{\lambda}h\}_{\gamma\in\Gamma, \lambda\in\Lambda}$ are Bessel families. Then the
following statements are equivalent:
\begin{enumerate}[(i)]
 \item $\{E_{\gamma}T_{\lambda}g\}_{\gamma\in\Gamma, \lambda\in\Lambda}$ and $\{E_{\gamma}T_{\lambda}h\}_{\gamma\in\Gamma, \lambda\in\Lambda}$ are dual frames for $L^2(G)$,
 \item for each $\alpha \in \LG^{\perp}$ we have
\[
  \int_{\Lambda}
 \overline{g(x-\lambda)} h(x-\lambda+\alpha ) \, d\mu_{\Lambda}(\lambda) =
\delta_{\alpha,0} \quad \almoste\ x \in G, 
\]
 \item for each $\beta \in \Lambda^{\perp}$ we have
\[
  \int_{\Gamma}
 \overline{\hat{g}(\omega\gamma)} \hat{h}(\omega\gamma\beta ) \, d\mu_{\Gamma}(\gamma) =
\delta_{\beta,1} \quad \almoste\ \omega \in \ghat. 
\]
\end{enumerate}
\end{theorem}
\begin{proof} By Remark~\ref{rem:stand-hypo-special-cases} the standing hypotheses are satisfied by the Gabor system. The result now follows from Theorem~\ref{th:TI-systems} together with Lemma \ref{thm:unitary-equi-duals} and the comments preceding Theorem~\ref{th:dual-Gabor}.
\end{proof}

From Theorem \ref{th:dual-Gabor} we can derive numerous results about Gabor systems.  We begin with
an example concerning the inversion of the short-time Fourier transform.
\begin{example} Let $g,h\in L^2(G)$ and consider $\{E_{\gamma}T_{\lambda}g\}_{\gamma\in\ghat,\lambda\in G}$ and $\{E_{\gamma}T_{\lambda}h\}_{\gamma\in\ghat,\lambda\in G}$. We equip $G$ and $\ghat$ with their respective Haar measures $\mu_{G}$ and $\mu_{\ghat}$. 
For $f\in L^2(G)$ we calculate
\begin{equation} \label{eq:0203d} \langle f, E_{\gamma}T_{\lambda} g \rangle = \int_{G} f(x) \overline{g(x-\lambda)}  \gamma(x) \, d\mu_{G}(x) = \mathcal F(f(\cdot)\overline{g(\cdot-\lambda)})(\gamma).
\end{equation}
With equation \eqref{eq:0203d} and since $\Vert f \Vert = \Vert \mathcal F f \Vert$, we find
\begin{align*}
 \int_G \int_{\ghat} \abs{\langle f, E_{\gamma}T_{\lambda} g \rangle }^2 \, d\mu_{\ghat}(\gamma)\, d\mu_{G}(\lambda) & = \int_G \int_{\ghat} \absbig{\mathcal F(f(\cdot)\overline{g(\cdot-\lambda)})(\gamma)}^2 \, d\mu_{\ghat}(\gamma)\, d\mu_{G}(\lambda) \\ & = \int_G \int_{G} \absbig{f(x)\overline{g(x-\lambda)}}^2 \, d\mu_{G}(x)\, d\mu_{G}(\lambda) \\
 & = \int_G \absbig{f(x)}^2  \int_{G} \absbig{ g(x-\lambda)}^2 \, d\mu_{G}(\lambda)\, d\mu_{G}(x) = \Vert f \Vert^2  \Vert g \Vert^2. 
\end{align*}
The same calculation holds for the Gabor system generated by $h$. We conclude that both Gabor systems are Bessel families. By Theorem \ref{th:dual-Gabor} the two Gabor systems 
$\{E_{\gamma}T_{\lambda}g\}_{\gamma\in\ghat,\lambda\in G}$ and $\{E_{\gamma}T_{\lambda}h\}_{\gamma\in\ghat,\lambda\in G}$ are dual frames for $L^2(G)$ if, and only if, for a.e. $x\in G$
\[ \int_G \overline{g(x-\lambda)}  h(x-\lambda) \, d\mu_G(\lambda) = \int_G \overline{g(\lambda)} h(\lambda) \, d\mu_G(\lambda) = \overline{\langle g, h\rangle } = 1, \]
that is, $\langle g, h\rangle = 1.$
This result is the well-known inversion formula for the short-time Fourier transform \cite{MR1601095,gr01}. 
\end{example}

\begin{example}
  Let $G=\Gamma=\R^n, \Lambda=\Z^n$ and $g\in L^2(\R^n)$. We equip $G$ and $\Gamma$ with the
  Lebesgue measure and $\Lambda$ with the counting measure.  Then
  \[
  \langle f_1,f_2\rangle = \int_{\R^n} \sum_{\lambda\in\Z^n} \langle f_1, E_{\gamma}T_{\lambda} g
  \rangle \langle E_{\gamma}T_{\lambda}g,f_2\rangle \, d\gamma, \quad \text{for all } f_1,f_2\in L^2(\R^n)
  \]
  if, and only if, \[\sum_{\lambda \in \Z^n} \abs{g(x-\lambda)}^2=1, \ \text{a.e. } x\in \R^n.\]
  Equivalently in the frequency domain, for all $\beta \in \Z^n$
  \[
  \int_{\R^n} \overline{\hat{g} (\omega+\gamma)}\hat{g}(\omega+\gamma+\beta) \, d\gamma =
  \delta_{\beta,0} \quad \text{a.e. } \omega\in\R^n.
  \] 
  From the time domain characterization, it is clear that the square-root of any uniform B-splines can be used to
  construct such functions $g$.
  The Gabor system with $\Lambda=\R^n$ and $\Gamma=\Z^n$ has similar characterizing equations, see \cite[Example 2.1(b)]{MR2066831}.
\end{example}

\begin{example} 
\label{ex:gabor-L2}
  Let $g,h\in L^2(\R)$ and $a,b>0$ be given. Take $\Lambda = a\Z$ and $\Gamma=b\Z$. We equip $\R$
  with the Lebesgue measure and $\Lambda^{\perp}\cong\tfrac{1}{a}\Z$,
  $\Gamma^{\perp}\cong\tfrac{1}{b} \Z$ with the counting measure. From this follows that the
  measure on $\Lambda$ and $\Gamma$ is the counting measure multiplied with $a$ and $b$
  respectively. Theorem
  \ref{th:dual-Gabor} now yields the following characterizing equation for dual Gabor systems in
  $L^2(\R)$: If $\{E_{\gamma}T_{\lambda}g\}_{\gamma\in\Gamma, \lambda\in\Lambda}$ and
  $\{E_{\gamma}T_{\lambda}h\}_{\gamma\in\Gamma, \lambda\in\Lambda}$ are Bessel sequences, then
  \[
  f = ab \sum_{\lambda\in a\Z} \sum_{\gamma\in b\Z} \langle f, E_{\gamma}T_{\lambda} g\rangle
  E_{\gamma}T_{\lambda} h, \quad \text{for all } f\in L^2(\R)
  \]
  if, and only if, for all $\alpha \in \tfrac{1}{b} \Z$
  \[
  \sum_{\lambda\in a\Z} \overline{g(x-\lambda)}  h(x-\lambda+\alpha) = \tfrac{1}{a}
  \delta_{\alpha,0} \quad \text{for a.e. } x\in [0,a].
  \]
  This result is equivalent to the characterization result by
  Janssen~\cite{MR1601115}. Higher dimensional versions can be derived similarly; see Ron and Shen~\cite{MR1460623} for alternative proofs. 
\end{example}

One can easily deduce characterization results for Gabor systems in $\ell^2(\Z^d)$ following the approach of the
preceding example. We refer to the work of Janssen~\cite{MR1601119} and Lopez and
Han~\cite{MR3091773} for direct proofs. Finally, we mention the following characterization for
finite and discrete Gabor frames.
\begin{example}
\label{ex:gabor-finite}
  Let $g,h\in \C^d$ and $a,b,d,N,M\in\N$ be such that $aN=bM=d$. Then
  \[ f = \sum_{m=0}^{M-1} \sum_{n=0}^{N-1} \langle f, E_{mb}T_{na}g \rangle E_{mb}T_{na}h, \quad \text{for all } f\in\C^d \]
  if, and only if,
 \[ 
 \sum_{k=0}^{N-1} \overline{g(x\!-\!nM\!-\!ka)}h(x\!-\!ka) = \frac{1}{M} \delta_{n,0} , \quad \forall
 x\!\in\!\{0,1,\ldots,a\!-\!1\}, n\in\{0,1,\ldots,b\!-\!1\}.
 \]
 This result appears first in\cite{ISI:A1995RF41200008} and has been rediscovered in, e.g., \cite{MR3019770}. 
\end{example}

\subsection{Wavelet and shearlet systems}
\label{sec:wavelets-and-shearlets}

Following \cite{BowRos2014}, we let $\epick{G}$ denote the semigroup of continuous group homomorphisms $\epickdila$ of $G$ onto $G$ with compact kernel. This semigroup can be viewed as an extension of the group of topological automorphisms on $G$; we define the extended \emph{modular function} $\Delta$ in $\epick{G}$ as in \cite[Section 6]{BowRos2014}. The isometric dilation operator $D_\epickdila: L^2(G) \to L^2(G)$ is then defined by
\[ D_\epickdila f(x) = \Delta{(\epickdila)}^{-1/2}f(\epickdila (x)).\]

 Let $\mathcal{A}$ be  a subset of $\epick{G}$, let $\Gamma$ be a co-compact subgroup of $G$, and let $\Psi$ be a subset of $L^2(G)$. The wavelet system generated by $\Psi$ is:
\[
  \mathrm{W}(\Psi,\mathcal{A},\Gamma):= \setprop{D_\epickdila T_\gamma \psi}{\epickdila \in \mathcal{A},\gamma \in \Gamma,\psi \in \Psi}.
\]
Depending on the choice of $\mathcal{A}$ and the structure of $\epick{G}$, it might be desirable to extend the wavelet system with translates of ``scaling'' functions, that is, $\setprop{T_\gamma \phi}{\gamma \in \Gamma, \phi \in \Phi}$ for some $\Phi \subset L^2(G)$. We denote this extension to a ``non-homogeneous'' wavelet system by $\mathrm{W}_h(\Psi,\Phi,\mathcal{A},\Gamma)$. If $\epick{G}$ only contains trivial group homomorphisms, \eg as in the case of $G=\Z$, it is possible to define the dilation operator on the dual group $\ghat$ via the Fourier transform.

The two wavelet systems introduced above offer a very general setup that include most of the usual wavelet-type systems in $L^2(\R^n)$, \eg  discrete and continuous wavelet and shearlet systems~\cite{Daubechies:1992:TLW:130655,MR2896273} as well as composite wavelet systems.

\begin{example}
  \label{ex:entire-group-wavelets}
 Let us consider the general setup as above, where we make the specific choice $\Gamma=G$ and $\Psi=\set{\psi_j}_{j \in J}$ for some index set $J \subset \Z$. For $\epickdila \in \mathcal{A}$ and $\gamma \in \Gamma=G$, we have
\[ D_\epickdila T_\gamma \psi_j (x)= \Delta(\epickdila)^{-1/2} \psi_j(\epickdila(x)-\gamma))=T_{\bar{\gamma}}D_\epickdila \psi_j(x) \]
for some $\bar{\gamma} \in \epickdila^{-1}\Gamma$ so that $\epickdila(\bar{\gamma})=\gamma$. It follows that $\mathrm{W}(\Psi,\mathcal{A},\Gamma)$ is a (generalized) translation invariant system for $\Gamma_j=G$ with $j\in J$ and $g_{j,p}=g_{j,\epickdila}=D_\epickdila \psi_j$ for $(j,p)=(j,\epickdila) \in J \times \mathcal{A}$. For simplicity we equip each measure space $P_j=\mathcal{A}$, $j \in J$, with the same measure; as usual we require that this measure $\mu_{\mathcal{A}}$ satisfies our standing hypotheses. Further, we define the adjoint of $\epickdila$ by $\hat{\epickdila}(\omega)=\omega \circ \epickdila$ for $\omega \in \ghat$. Using results from \cite{BowRos2014}, it follows that $\hat{\epickdila}$ is an isomorphism from $\ghat$ onto $(\ker {\epickdila})^\perp$ and that 
\[ \widehat{D_\epickdila f}(\omega) =
\begin{cases}
  \Delta(\epickdila)^{1/2} \hat{f}(\hat{\epickdila}^{-1}(\omega)) & \omega \in {(\ker \epickdila)}^\perp, \\
 0 & \text{otherwise}.
\end{cases}
\]

As translation invariant systems always satisfy the local integrability condition, we immediately have that  $\mathrm{W}(\Psi,\mathcal{A},G)$ is a Parseval frame, that is,
\[ f = \sum_{j\in J} \int_{\mathcal{A}} \int_{G} \innerprod{f}{D_\epickdila T_\gamma \psi_j} D_\epickdila T_\gamma \psi_j \, d\mu_{G}(\gamma) \, d\mu_{\mathcal{A}}(\epickdila) \quad \text{for all } f \in L^2(G), \]
if, and only if, for a.e. $\omega \in \ghat$, 
\begin{equation}
 t_0 =\sum_{j\in J} \int_{\mathcal{A}} \absbig{\widehat{D_\epickdila \psi_j}(\omega)}^2 d\mu_{\mathcal{A}}(\epickdila)= \sum_{j\in J}\int_{\{\epickdila \in \mathcal{A} \, : \, \omega \in (\ker \epickdila)^\perp\}} \Delta(\epickdila) \absbig{\hat \psi_j(\hat{\epickdila}^{-1}(\omega))}^2 d\mu_{\mathcal{A}}(\epickdila) = 1 .\label{eq:gen-Calderon} 
\end{equation}
In particular, it follows that $\mathrm{W}(\Psi,\mathcal{A},G)$ cannot be a Parseval frame for $L^2(G)$ regardless of the measure $\mu_\mathcal{A}$ if $\ghat \setminus \cup_{\epickdila \in \mathcal{A}} (\ker \epickdila)^\perp$ has non-zero measure.

The Calder\'on admissibility condition~(\ref{eq:t-alpha-cont-wavelet-intro}) is a special case of
(\ref{eq:gen-Calderon}).  To see this, take $G=\R$ and consider the dilation group
$\mathcal{A}=\setprop{x \mapsto a^{-1} x}{a \in\R \setminus\set{0}}$ with measure
$\mu_{\mathcal{A}}$ defined on the Borel algebra on $\R \setminus\set{0}$ by
$d\mu_{\mathcal{A}}(a)=da/a^2$, where $da=d\lambda(a)$ denotes the Lebesgue measure. Higher
dimensional versions of Calder\'on's admissibility condition are obtained similarly, see also
\cite{MR2652610,MR1881293}.
\end{example}

\begin{example}
  \label{ex:wavelets-in-Rn}
  We consider wavelet systems in $L^2(\R^n)$ with discrete dilations and semi-continuous
  translations. Let $A \in \textrm{GL}(n,\R)$ be a matrix whose eigenvalues are strictly larger than
  one in modulus, set $\mathcal{A}=\setprop{x\mapsto A^jx}{j \in \Z}$, and let $\Gamma$ be a
  co-compact subgroup of $\R^n$. The wavelet system generated by $\Psi=\{\psi_\ell\}_{\ell=1}^L
  \subset L^2(G)$ is given by
  \[ \mathrm{W}(\Psi,A,\Gamma):=\setprop{D_{A^j}T_\gamma
    \psi_\ell=\abs{\det{A}}^{-j/2}\psi_\ell(A^{-j}\cdot-\gamma)}{\ell=1,\dots,L,j \in \Z, \gamma \in
    \Gamma}.\] Any co-compact subgroup of $\R^n$ is of the form $\Gamma=P(\Z^k \times \R^{n-k})$ for
  some $k \in \set{0,1,\dots,n}$ and $P \in \mathrm{GL}(n,\R)$. Since
  $\mathrm{W}(\{\psi\},A,\Gamma)$ is unitarily equivalent to
  $\mathrm{W}(\{D_{P^{-1}}\psi\},P^{-1}AP,\Z^k \times \R^{n-k})$ we can without loss of generality
  assume that $P=I_n$, \ie $\Gamma=\Z^k \times \R^{n-k}$.

  Clearly, $\mathrm{W}(\Psi,A,\Gamma)$ is a generalized translation invariant system for
  $\Gamma_j=A^{j}\Gamma$ with $j\in J:=\Z$ and $g_{j,\ell}=D_{A^j} \psi_\ell$, where
  $P_j=\set{1,\dots,L}$. To get rid of a scaling factor in the representation formula, we will use
  $\mu_{P_j}=\frac{1}{\abs{\det{A}}^j} \mu_c$ as measure on $P_j=\set{1,\dots,L}$, where $\mu_c$
  denotes the counting measure. The standing assumptions are clearly satisfied. Moreover, the local
  integrability condition is known to be equivalent to local integrability on $\R^n\setminus \{0\}$
  of the Calder\'on sum \cite[Proposition 2.7]{MR2746669} and can, therefore, be omitted from the
  characterization results.  It follows that two Bessel families $\mathrm{W}(\Psi,A,\Gamma)$ and
  $\mathrm{W}(\Phi,A,\Gamma)$ are dual frames if, and only if, with $B=A^T$,
  \[
  t_\alpha(\omega) = \sum_{l=1}^L \sum_{\sumstyle{j\in\Z}{\alpha \in B^j(\Z^k \times \{0\}^{n-k})}}
  \hat{\psi}_l(B^{-j}\omega) \overline{\hat{\phi}_l(B^{-j}(\omega+\alpha))} =\delta_{\alpha,0}
  \quad\text{for \almoste\ }\omega \in \R^n,
 \] 
 for all $\alpha \in \Z^k \times \{0\}^{n-k}$. For $k=n$ this result was obtained in
 \cite{MR1891728}, extending the work of Gripenberg~\cite{MR1338828} and Wang~\cite{MR2692675}.
\end{example}

\begin{example}
  \label{sec:shearlets}
  Let us finally consider the cone-adapted shearlet systems. For brevity we
  restrict our findings to the non-homogeneous, continuous shearlet
  transform in dimension two. Let
  \[ A_a =
  \begin{pmatrix}
    a & 0 \\ 0 & a^{1/2}
  \end{pmatrix} \quad \text{and} \quad S_s =
  \begin{pmatrix}
    1 & s \\ 0 & 1
  \end{pmatrix}
  \]
  for $a\neq 0$ and $s \in \R$.  For $\psi \in L^2(\R^2)$ define 
  \[ \psi_{ast}(x):=a^{-3/4} \psi(A_a^{\, -1} S_s^{\, -1}(x-t)  ) =T_t
  D_{S_s A_a} \psi.  \] The cone-adapted continuous shearlet system
  $\mathrm{S}_h(\phi,\psi,\tilde{\psi})$ is then defined as the
  collection:
  \begin{multline*}
    \mathrm{S}_h(\phi,\psi_1,\psi_2) = \setpropbig{T_t \phi}{t \in \R^2}
    \cup \setprop{T_t D_{S_s A_a} \psi_1}{a \in \itvoc{0}{1},
      \abs{s}\le 1+a^{1/2},t \in \R^2} \\ \cup \setprop{T_t
      D_{\tilde{S}_s \tilde{A}_a} \psi_2}{a \in \itvoc{0}{1},
      \abs{s}\le 1+a^{1/2},t \in \R^2},
  \end{multline*}
  where $\tilde{S}_s=S_s^{\,T}$ and
  $\tilde{A}_a=\diag{a^{1/2},a}$. This is a special case of the system $\mathrm{W}_h$ introduced above. More importantly, this is a GTI system. To see
  this claim, take $J =\set{0,1}$ and $\Gamma=\Gamma_{\! j}=\R^2$ for
  $j \in J$. Define $P_0=\set{0}$ and let $\mu_{P_0}$ be the counting
  measure on $P_0$. Define
  \[ P_1 = \setprop{(a,s)\in \R^2}{a \in \itvoc{0}{1}, \abs{s}\le
    1+a^{1/2}},\] and let $\mu_{P_1}$ be some measure on $P_1$ so that
  our standing hypotheses are satisfied. The generators are
  $g_{0,p}=g_{0,0}=\phi$ for $p=0\in P_0$ and
  $g_{1,p}=g_{1,(a,s)}=D_{\tilde{S}_s \tilde{A}_a} \psi$ for
  $p=(a,s)\in P_1$. This proves our claim. By
  Theorem~\ref{th:TI-systems} and Lemma~\ref{th:TI-systems-along-G} we
  immediately have that, if $\mathrm{S}_h(\phi,\psi_1,\psi_2)$ and
  $\mathrm{S}_h(\phi,\tilde \psi_1,\tilde\psi_2)$ are Bessel families,
  then they are dual frames if, and only if,
  \begin{multline}
    \overline{\hat\phi(\omega)}\hat\phi(\omega) + \int_{P_1}a^{3/2}
    \overline{\hat{\psi}_1(A_aS_s^T\omega)}
    \hat{\tilde{\psi}}_1(A_aS_s^T\omega) \, d\mu_{P_1}(a,s)\\ +
    \int_{P_1}a^{3/2} \overline{\hat{\psi}_2(\tilde A_a \tilde
      S_s^T\omega)} \hat{\tilde{\psi}}_2(\tilde A_a \tilde
    S_s^T\omega) \, d\mu_{P_1}(a,s) = 1 \qquad \text{for a.e.\ $\omega
      \in \R^2$}.
    \label{eq:shearlet}
  \end{multline}
  A standard choice for the measure $\mu_{P_1}$ in~(\ref{eq:shearlet})
  is $d\mu_{P_1}(a,s)=\frac{dads}{a^3}$, which comes from the
  left-invariant Haar measure on the shearlet group. The above
  characterization result generalizes results from \cite{MR2471937,MR2803946,MR2861594}.
\end{example}

\subsection{Other examples}
\begin{example} 
\label{ex:p-adic}
In this example we consider the additive group of $p$-adic integers $\mathbb{I}_p$. To introduce
this group, we first consider the $p$-adic numbers $\Q_p$. Here $p$ is some fixed prime-number. The
$p$-adic numbers are the completion of the rationals $\Q$ under the $p$-adic norm, defined as
follows. Every non-zero rational $x$ can be uniquely factored into $x=\tfrac{r}{s} p^n$, where
$r,s,n\in\Z$ and $p$ does not divide $r$ nor $s$. We then define the $p$-adic norm of $x$ as $\Vert
x \Vert_p =p^{-n}$, additionally $\Vert 0 \Vert_p := 0$. The $p$-adic numbers $\Q_p$ are the
completion of $\Q$ under $\Vert\!\cdot\!\Vert_p$. It can be shown that all $p$-adic numbers $x$ can
be written uniquely as
\begin{equation} 
  \label{eq:p-adic-numbers} 
  x = \sum_{j=k}^{\infty} x_jp^j,
\end{equation}
where $x_k\in \{0,1,\dots, p-1\}$ and $k\in\Z$, $x_k\ne 0$. The set of all numbers $x\in\Q_p$
for which $x_j=0$ for $j<0$ in \eqref{eq:p-adic-numbers} are the $p$-adic
integer $\mathbb{I}_p$. Equivalently, $\mathbb{I}_p = \{ x\in\Q_p \, : \, \Vert x \Vert_p \le 1\}$. In fact,
$\mathbb{I}_p$ is a compact, closed and open subgroup of $\Q_p$.  Its dual group $\widehat{\mathbb{I}}_p$ can be
identified with the Pr\"ufer $p$-group $\Z(p^{\infty})$, which consists of the union of the
$p^n$-roots of unity for all $n\in\N$. That is,
\begin{equation*} 
  \widehat{ \mathbb{I}_p} \cong \Z(p^{\infty}) := \{ e^{2\pi i m/p^n} \, : \, n\in\N,m\in\{0,1,\ldots, p^n-1\}\, \} \subset \C.
\end{equation*}
We equip $\Z(p^{\infty})$ with the discrete topology and multiplication as group operation.  For
more information on $p$-adic numbers and their dual group we refer to, \eg \cite[\S 10, \S 25]{MR0156915}. For
$n\in\N$ consider now the subgroups $\Gamma_{\! n}^{\perp} = \{ e^{2\pi i m/p^n} \, : \,
m=0,1,\ldots,p^n-1\} \subset \Z(p^{\infty})$.  Note that all $\Gamma_{\! n}^{\perp}$ are finite groups
of order $p^n$ and generated by $e^{2\pi i /p^n}$. Moreover, all $\Gamma_{\! n}^{\perp}$ are nested so that
\[ 
1 \subset \Gamma_{\! 1}^{\perp} \subset \Gamma_{\! 2}^{\perp} \subset \, \cdots \, \subset
\Z(p^{\infty}). 
\] 
Let now $\{g_n\}_{n\in\N} \subset L^2(\mathbb{I}_p)$. By Theorem~\ref{th:compact-G-dual} the
generalized translation invariant system $\{ T_{\gamma} g_n\}_{\gamma\in\Gamma_{\! n},n\in\N}$ is a
Parseval frame for $L^2(\mathbb{I}_p)$ if, and only if, for each $\alpha\in \bigcup_{n\in\N}
\Gamma_{\! n}^{\perp} = \Z(p^{\infty})$
\[ 
\sum_{k=n^*}^{\infty} \overline{ \hat{g}_n(\omega) } \hat{g}_n(\omega\alpha) = \delta_{\alpha,1}
\quad \text{for all } \omega\in \Z(p^{\infty}), 
\]
where $n^*\in\N$ is the smallest natural number such that $\alpha\in \Gamma_{\! n^*}^{\perp}$.
Because we consider a GTI system with countably many generators, the standing hypotheses are
trivially satisfied, see Section \ref{sec:gsi}.

Returning to the $p$-adic numbers $\Q_p$, we note that the only co-compact subgroup of $\Q_p$ is
$\Q_p$ itself \cite{BowRos2014}. Therefore any GTI system in $L^2(\Q_p)$ is, in fact, a translation
invariant system of the form $\cup_{j \in J}\set{T_\gamma g_{p}}_{\gamma \in \Q_p, p \in P_j}$. The
equations characterizing the dual frame property of such systems are immediate from
Theorem~\ref{th:TI-systems} and Lemma~\ref{th:TI-systems-along-G}.

Finally, in the product group $\Q_p \times \mathbb{I}_p$ there are no discrete, co-compact
subgroups~\cite{BowRos2014}, and thus no generalized shift invariant systems for $L^2(\Q_p \times
\mathbb{I}_p)$ can be constructed. However, any subgroup of the form $\Q_p \times \Gamma_{\! n}$,
where $\Gamma_{\! n}$ is a co-compact subgroup of $\mathbb{I}_p$ as before, is a co-compact subgroup
in $\Q_p\times\mathbb{I}_p$, indicating that a large number of generalized translation invariant
systems do exist in $L^2(\Q_p\times \mathbb{I}_p)$.
\end{example}

In order to apply Theorem \ref{th:2001e} to a given GTI system, one needs to verify that the (dual)
$\alpha$-LIC or the stronger LIC holds. By Theorem \ref{th:TI-systems} we get this for free for
translation invariant systems. For regular wavelet systems as in Example~\ref{ex:wavelets-in-Rn} the
LIC has an easy characterization~\cite[Proposition 2.7]{MR2746669}. For certain irregular wavelet
systems over the real line a detailed analysis of the LIC has been carried out in \cite{MR2211335}
using Beurling densities. However, for general GTI systems there is no simple interpretation of the
local integrability conditions.

\section*{Acknowledgments}
We thank O.\ Christensen for giving access to an early version of \cite{OleSaySon2014} and K.\ Ross
for providing an example of an LCA group with no uniform lattices, but with proper co-compact
subgroups. 

\def\cprime{$'$} \def\cprime{$'$} \def\cprime{$'$}
  \def\uarc#1{\ifmmode{\lineskiplimit=0pt\oalign{$#1$\crcr
  \hidewidth\setbox0=\hbox{\lower1ex\hbox{{\rm\char"15}}}\dp0=0pt
  \box0\hidewidth}}\else{\lineskiplimit=0pt\oalign{#1\crcr
  \hidewidth\setbox0=\hbox{\lower1ex\hbox{{\rm\char"15}}}\dp0=0pt
  \box0\hidewidth}}\relax\fi} \def\cprime{$'$} \def\cprime{$'$}
  \def\cprime{$'$} \def\cprime{$'$} \def\cprime{$'$} \def\cprime{$'$}

\end{document}